\documentclass[11pt,a4paper]{amsart}
\usepackage[english]{babel}
\usepackage{graphicx}
\usepackage{framed}
\usepackage[normalem]{ulem}
\usepackage{amsmath}
\usepackage{amsthm}
\usepackage{amssymb}
\usepackage{amsfonts}
\usepackage[foot]{amsaddr}
\usepackage{mathrsfs,amsmath}
\usepackage{enumerate}
\usepackage[utf8]{inputenc}
\usepackage[top=1 in,bottom=1in, left=1 in, right=1 in]{geometry}
\usepackage{mathtools}
\usepackage{tikz-cd}
\usepackage[font=small,labelfont=bf]{caption} 
\usepackage[subrefformat=parens]{subcaption}
\usepackage{cite}
\usepackage{abstract}

\usetikzlibrary{decorations.markings, arrows.meta}

\tikzset{
  marrow/.style={decoration={markings,mark=at position 0.5 with {\arrow{#1}}}, postaction=decorate}
}

\usepackage{enumitem}
\setlist{  
  listparindent=\parindent,
  parsep=0pt,
}

\usepackage{bbm}

\usepackage{stackrel}

\newcommand{\supp}{{\rm supp}}

\newcommand{\R}{\mathbb{R}}
\newcommand{\N}{\mathbb{N}}

\newcommand{\Z}{\mathbb{Z}}

\newcommand{\eps}{\varepsilon}

\newtheorem{thm}{Theorem}[section]
\newtheorem{cor}[thm]{Corollary}
\newtheorem{lem}[thm]{Lemma}
\newtheorem{prop}[thm]{Proposition}

\theoremstyle{definition}

\theoremstyle{remark}

\newtheorem{rmk}[thm]{Remark}

\newcommand{\bproof}{\noindent{\textit{Proof. }}}
\newcommand{\eproof}{\hfill\qed}

\newcommand{\bproofof}[1]{\noindent{\textit{Proof of #1. }}}

\newcommand{\Cont}{{\rm Cont_0}}

\begin{document}

\title{Small energy isotopies of loose Legendrian submanifolds}
\author{Lukas Nakamura}
\maketitle

\begin{abstract}
	We prove that for a closed Legendrian submanifold $L$ of dimension $n \geq 2$ with a loose chart of size $\eta$, any Legendrian isotopy starting at $L$ can be $C^0$-approximated by a Legendrian isotopy with energy arbitrarily close to $\frac{\eta}{2}$. This in particular implies that the displacement energy of loose displaceable Legendrians is bounded by half the size of its smallest loose chart, which proves a conjecture of Dimitroglou Rizell and Sullivan \cite{drs20}.\\
\end{abstract}

\section{Introduction}

In many situations it requires a positive amount of energy to connect two different Legendrian submanifolds of a contact manifold via a contact isotopy. On the other hand, Murphy \cite{mur12} showed that for loose Legendrians the existence of a contact isotopy is a purely topological question. The goal of this paper is to give an upper bound for the  minimal energy that is required for Legendrian isotopies of loose Legendrians.

Let $(M,\alpha)$ be a strict contact manifold of dimension $2n+1$, and let $L$ be a closed Legendrian submanifold. This means that $\alpha$ is a 1-form on $M$ such that $\alpha \wedge (d \alpha)^n$ defines a volume form, and $L$ is everywhere tangent to $\xi \coloneqq \ker\alpha$ and of dimension $n$. The Reeb vector field $R_\alpha$ is the unique vector field on $M$ defined by $i_{R_\alpha} d \alpha = 0$ and $\alpha(R_\alpha)=1$. A Reeb chord of $L$ is a flow line $\gamma:[0,l] \to M$ of $R_\alpha$ with endpoints on $L$. $l > 0$ is called the \emph{action} of the Reeb chord.

We consider isotopies $L_t$, $t \in [0,1]$, of $L$ through Legendrian submanifolds. It is a general fact that such isotopies are always induced by an ambient contact isotopy $\phi_t$ of $M$, i.e. an isotopy of $M$ that preserves $\xi$ (in fact, this even holds for parametrized Legendrians, c.f. Theorem 2.6.2 in \cite{gei08}). We can associate to $\phi_t$ its contact Hamiltonian $H: M \times [0,1] \to \R$, which is defined by the formula
\begin{equation}
H(\phi_t(x),t)= \alpha(\dot \phi_t(x)),
\end{equation}
where $\dot \phi_t(x)$ denotes the time-derivative of $\phi_t(x)$. Conversely, given a function $H: M \times [0,1] \to \R$, which we may also view as a time-dependent function $H_t:M \to \R$, $t \in [0,1]$, we can define its time-dependent contact vector field $X_{H_t}$ via the equations
\begin{equation}
H_t = \alpha(X_{H_t}) \quad \quad \text{and} \quad \quad -dH_t|_{\xi} = i_{X_{H_t}} d \alpha|_{\xi}.
\end{equation}
The condition $\alpha \wedge (d \alpha)^n \neq 0$ ensures that this defines a unique vector field. The flow of $X_{H_t}$ preserves $\xi$, and thus gives a contact isotopy $\phi^H_t$. It is straightforward to check that these two correspondences between functions on $M$ and contact isotopies are inverse to each other.

To a contact isotopy $\phi_t$ and its associated Hamiltonian $H_t$ we associate the energy
\begin{equation}
\Vert \phi_t \Vert_\alpha = \Vert H_t \Vert \coloneqq \int_0^1 \underset{x \in M}{\text{max}}\, |H_s(x)| ds,
\end{equation}
 which induces a non-degenerate metric on the space of contactomorphisms \cite{she17}.

Unless stated otherwise, all manifolds here and below are assumed to be connected, and isotopies always start at the identity.\\

Assume that $L_0$ and $L_1$ are two distinct closed Legendrian submanifolds of $M$ Legendrian isotopic to each other. We are interested in the infimum of the energies of contact isotopies that move $L_0$ to $L_1$. Denote this infimum by $d(L_0,L_1)$. In \cite{rz18}, Rosen and Zhang showed that either $d(L_0,L_1) = 0$ or $d(L_0,L_1) > 0$ always holds for fixed $L_0$ independent of $L_1$. It is expected that the latter holds under quite general assumptions on $M$ (cf. Conjecture 1.10 in \cite{rz18}). For example, the following Theorem which combines results from \cite{drs20}, \cite{drs21}, and \cite{oh21} implies that this is indeed the case for displaceable Legendrians in contact manifolds which are either closed\footnote{The results of \cite{drs21} and \cite{oh21} also hold for more general classes of contact manifolds which may not be closed. } or of the form $(P \times \R, \lambda + dz)$, where $(P,d \lambda)$ is an exact geometrically bounded symplectic manifold and $z$ denotes the coordinate on $\R$. 

\begin{thm}\label{thm:leg energy capacity ineq}
	Let $(M,\alpha)$ be either  compact or equal to $(P \times \R, \lambda + dz)$ as above, and let $L_0, L_1$ be two distinct closed Legendrian submanifolds that can be connected by a Legendrian isotopy. If there are no Reeb chords between $L_0$ and $L_1$, then\footnote{The additional factor of 2 in Theorem \ref{thm:main result} when compared to the formulation of the results in \cite{drs20}, \cite{drs21}, or \cite{oh21} appears because in this paper the energy is measured using the maximum norm of a Hamiltonian, and not the oscillatory norm. By Remark \ref{rmk:comparison_oscillatory_and_maximums_norm}, the displacement energy defined in terms of the oscillatory norm is twice as large as the one used here as long as the Reeb vector field is complete.} $2 d(L_0,L_1)$ is bounded from below by the minimal action of Reeb chords of $L_0$ (and by symmetry also of $L_1$).
\end{thm}

In a strict contact manifold $(M,\alpha)$, a subset $A$ is said to be \emph{displaced} from a subset $B$ if there are no Reeb chords between $A$ and $B$. The \emph{displacement energy} of a Legendrian $L_0$ is the infimum of $d(L_0,L_1)$ over all Legendrians $L_1$ such that there are no Reeb chords between $L_0$ and $L_1$, i.e. Theorem \ref{thm:leg energy capacity ineq} states that the displacement energy of $L_0$ is bounded from below by half of the minimal action of Reeb chords of $L_0$.

We are concerned with the converse question. Can we give an upper bound on $d(L_0,L_1)$ depending on $L_0$ and $L_1$? As a first step, it was proven in \cite{drs20} that the displacement energy of the standard Legendrian 2-sphere in $\R^5$ can be made arbitrarily small by adding a stabilization contained in a sufficiently small neighbourhood of a point $x \in L$. The authors of \cite{drs20} conjectured that the same should hold for any closed Legendrian in a contact manifold. We will use their techniques to prove this conjecture if $\dim L \geq 2$, and, in fact, give an explicit bound of the displacement energy in terms of the size of the stabilization (Corollary \ref{cor:displacement_of_stabilized_legendrian}). It turns out that this upper bound coincides with the lower bound from Theorem \ref{thm:leg energy capacity ineq} for ``nice" stabilizations and therefore is optimal.

These results follow from the following more general theorem about loose Legendrians which states that we can guarantee the existence of an isotopy of small energy, and even $C^0$-approximate any given isotopy.

\begin{thm}\label{thm:main result}
	Let $(M^{2n+1},\alpha)$, $n \geq 2$, be a strict contact manifold, and let $U_0,U_1 \subseteq M$ be open subsets with the property that there exist $\eps_0,\eps_1 >0$ such that the energy (as a contactomorphism of $U_i$) of the time-1 map of any compactly supported contact isotopy $\psi^i_t:U_i \to U_i$ is smaller than $\eps_i$ for $i\in \{0,1\}$. Let $f_t: L \to M$, $t \in [0,1]$, be a homotopy of closed, connected Legendrian embeddings such that $f_i(L) \cap U_i$ is a loose Legendrian submanifold of $U_i$ for $i\in \{0,1\}$. Then, for any given $\eta > 0$, there exist compactly supported contact isotopies $\phi_t$ and $\psi^i_t, i \in \{0,1\}$, with $\Vert \phi_t \Vert_\alpha < \eta$, $\supp(\psi^i_t) \subseteq U_i$, $\Vert \psi^i_t \Vert_\alpha < \eps_i$, and $\psi^1_1 \circ \phi_1 \circ \psi^0_1 \circ f_0 = f_1$. Furthermore, given any $\delta > 0$, these isotopies can be chosen in such a way that $\phi_t\circ \psi^0_1 \circ f_0$ is $\delta$-close\footnote{Throughout this paper, closeness refers to strict $C^0$-closeness, i.e. two functions $f$ and $g$ are $\delta$-close if and only if $\Vert f - g \Vert_{C^0} < \delta$.} to $f_t$ for all $t \in [0,1]$ (see Figure \ref{fig:thm2.2}). In particular, the energy of the concatenation $(\psi^0 * \phi * \psi^1)_t$ is smaller than $\eps_0 + \eps_1 + \eta$.\\
\end{thm}

\begin{rmk}
	Proposition \ref{prop:jet-space_with_small_energies} gives an explicit class of examples of sets that satisfy the property of $U_0$ and $U_1$ in the statement of Theorem \ref{thm:main result}. In particular, any closed Darboux ball and thus also loose charts in the sense of \cite{mur12} satisfy this property for some $\eps > 0$. To be more precise, any open subset of a closed Darboux ball can be compressed into any arbitrarily small neighbourhood of the origin via contact isotopies with a bound on their energies depending only on the Darboux ball (to see this, consider the contact isotopy $(x,y,z) \mapsto (e^{-\lambda t}x,e^{-\lambda t}y, e^{-2\lambda t}z), t \in [0,1],$ on $(\R^{2n+1}, f(dz - y dx))$ for some function $f: \R^{2n+1} \to (0, \infty)$, and proceed as in the proof of Proposition \ref{prop:jet-space_with_small_energies}, using that $f$ is bounded when restricted to the Darboux ball). Following the proof of Proposition of \ref{prop:jet-space_with_small_energies}, we may thus assume that the Darboux ball is strict after possibly shrinking it, and then Proposition \ref{prop:jet-space_with_small_energies} applies.\\
\end{rmk}

The main ingredients in the proof of Theorem \ref{thm:main result} are the following four facts: (1) $d(\widetilde{L}_0,\widetilde{L}_1) = 0$ whenever $\widetilde{L}_0$ and $\widetilde{L}_1$ are two $n$-dimensional non-Legendrian submanifolds that can be connected via a contact isotopy \cite{rz18} and the refinement of this result to parametrized non-Legendrian submanifolds \cite{drs22}  (see Theorem \ref{thm:hofer_distance_for_non-legednrians}, Theorem \ref{thm:hofer distance for parametrized non-legendrians}, and Corollary \ref{cor:c^0_close_hofer_distance_for_non-legendrians} below), (2) any formally Legendrian submanifold can be $C^0$-approximated by loose Legendrians \cite{mur12} (Lemma \ref{lem:loose_legendrians_are_dense}), (3) Murphy's h-principle for loose Legendrians (Theorem \ref{thm:murphys h-principle}), and (4) upper bounds on the energy of contact isotopies in Weinstein neighbourhoods \cite{drs20} (Proposition \ref{prop:jet-space_with_small_energies}). In outline, the proof goes as follows. Let $\Phi_t$ be a contact isotopy so that $f_t = \Phi_t \circ f_0$. First $C^0$-perturb $f_0$ to a formal Legendrian embedding $g$ which is non-Legendrian. By (1), we can find a contact isotopy $\widetilde{\phi}_t$ with arbitrarily small energy so that $\Phi_1 \circ g = \widetilde{\phi}_1 \circ g$. Then $C^0$-approximate $g$ by a loose Legendrian embedding $\chi$ using (2). Let $h_i, i \in \{0,1\},$ be loose Legendrian embeddings obtained by stabilizing $f_i$ inside a of a small neighbourhood of a point in $U_i \cap f_i(L)$. We can perform these steps in such a way that $h_0$ is formally isotopic to $\chi$ inside of a small Weinstein neighbourhood of $f_0(L)$ and formally isotopic to $f_0$ via an isotopy with compact support in $U_0$, and $h_1$ is formally isotopic to $\widetilde{\phi}_1 \circ \chi$ inside of a small Weinstein neighbourhood of $f_1(L)$ and formally isotopic to $f_1$ via an isotopy with compact support in $U_1$. By Murphy's h-principle we can find a contact isotopies $\theta^i_t$ and $\psi^i_t, i \in \{0,1\},$ so that $\psi^i_t$ has compact support in $U_i$, $\psi^0_1 \circ f_0 = h_0$, $\psi^1_1 \circ h_1 = f_1$, $\theta^0_1 \circ h_0 = \chi$, and $\theta^1_1 \circ \widetilde{\phi}_1 \circ \chi = h_1$. $\psi^i_t$ can be chosen to satisfy $\Vert \psi^i_t \Vert_\alpha < \eps_i$ by assumption, and $\Vert \theta^i_t \Vert_\alpha$ can be assumed to be arbitrarily small by (4). Then the isotopies $\psi^i_t, i \in \{0,1\},$ and $\phi_t \coloneqq (\theta^0 * \widetilde{\phi} * \theta^1)_t$ have the desired properties and can, in fact, be chosen so that $\phi_t \circ \psi^0_1 \circ f_0$ is $C^0$-close to $f_t$. The details of the proof are explained in Section \ref{sec:proofs}.\\

As a consequence of Theorem \ref{thm:main result}, we obtain an upper bound on the displacement energy of a loose Legendrian.

\begin{thm}\label{thm:displacement of loose legendrians}
	Let $(M^{2n+1},\alpha)$, $n \geq 2$, be a strict contact manifold, and let $U \subseteq M$ be an open subset with the property that there exists an $\eps >0$ such that the energy (as a contactomorphism of $U$) of the time-1 map of any compactly supported contact isotopy $\psi_t:U \to U$ is smaller than $\eps$. Let $L \subseteq M$ be a closed, displaceable Legendrian submanifold such that $L \cap U $ is loose in $U$. Then, for any given $\eta > 0$ and $E>0$, there exist compactly supported contact isotopies $\phi_t$ and $\psi_t$ with $\Vert \phi_t \Vert_\alpha < \eta$, $\supp(\psi_t) \subseteq U$, and $\Vert \psi_t \Vert_\alpha < \eps$ such that all Reeb chords between $L$ and $\phi_1(\psi_1(L))$ have action larger than $E$. 
	
	If in addition the image of $L$ under the Reeb flow is closed as a subset of $M$, we may choose $\phi_t$ and $\psi_t$ as above such that there are no Reeb chords between $L$ and $\phi_1(\psi_1(L))$. In particular, the displacement energy of $L$ is not larger than $\eps$.\\
\end{thm}

\begin{figure}
	\centering
	\includegraphics[width=\textwidth,height=0.9\textheight,keepaspectratio]{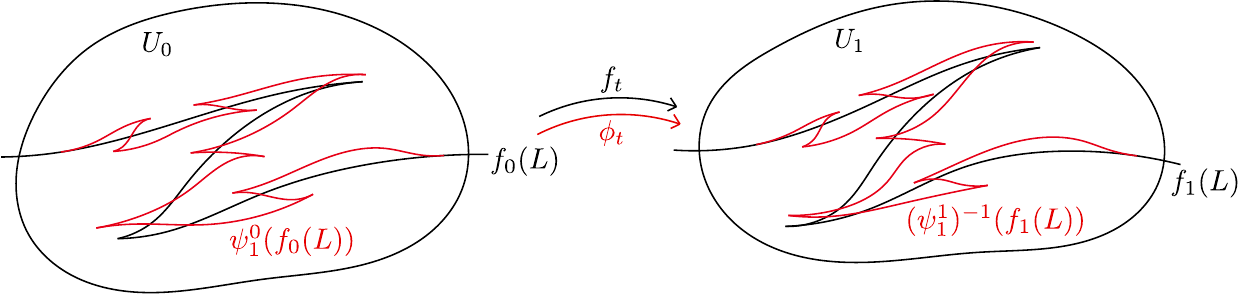}
	\caption{After perturbing the Legendrians there exists an isotopy $\phi_t$ of small energy.}
	\label{fig:thm2.2}
\end{figure}

In the case that $L$ may not be displaceable, we have the following more general statement.

\begin{thm}\label{thm:displacement_below_action_cutoff}
	Let $M$ and $U$ be as in Theorem \ref{thm:displacement of loose legendrians}, and let $E_1,E_2 \in \R$. Let $L \subseteq M$ be a closed Legendrian submanifold such that $L \cap U $ is loose in $U$, and assume that there exists a Legendrian submanifold $L_1$ that is Legendrian isotopic to $L$ such that all Reeb chords from $L$ to $L_1$ have action larger than $E_1$, and all Reeb chords from $L_1$ to $L$ have action larger than $E_2$. Then, for any given $\eta > 0$, there exist compactly supported contact isotopies $\phi_t$ and $\psi_t$ with $\Vert \phi_t \Vert_\alpha < \eta$, $\supp(\psi_t) \subseteq U$, and $\Vert \psi_t \Vert_\alpha < \eps$ such that all Reeb chords from $L$ to $\phi_1(\psi_1(L))$ have action larger than $E_1$, and all Reeb chords from $\phi_1(\psi_1(L))$ to $L$ have action larger than $E_2$.
	
	If in addition the image of $L$ under the non-negative (resp. non-positive) Reeb flow is closed as a subset of $M$, the above statement still holds if we allow $E_1 = \infty$ (resp. $E_2 = \infty$), meaning that there are no Reeb chords from $L$ to $L_1$ (resp. from $L_1$ to $L$). 
\end{thm}

We show in Section $\ref{sec:proof of stabilized case}$ how these results follow from Theorem \ref{thm:main result}.\\

\begin{rmk}
	Note that in a contactization $(M \times \R, dz + \lambda)$ of an exact symplectic manifold $(M,d \lambda)$ the image of any compact set under the Reeb flow is closed.\\
\end{rmk}

\begin{rmk}
	If $f_t:L \to (M, \alpha)$ is a homotopy of Legendrian embeddings of a compact, connected manifold $L$ of dimension $\geq 2$ with non-empty boundary, then for any $\eps > 0$ there exists a contact isotopy $\phi_t:M \to M$ such that $\phi_1 \circ f_0 = f_1$ and $\Vert \phi_t \Vert_\alpha< \eps$ (for unparametrized Legendrians see \cite{rz18}). Furthermore, $\phi_t \circ f_0$ can be chosen to $C^0$-approximate $f_t$. This follows from the same techniques as Theorem \ref{thm:main result} since $L_0$ and $L_1$ have arbitrarily small loose charts near their respective boundaries (where ``small" refers both to the diameter and the size of the loose chart as defined at the end of this section). Indeed, if $f:L \to M$ is a Legendrian embedding with non-empty boundary and $\xi_t: L \to L, t \in [0,1],$ is a homotopy of embeddings $C^0$-close to the identity starting at the identity so that $\xi_1(L)$ is contained in the interior of $L$, then for any stabilization $Sf:L \to M$ of $f$ sufficiently close the boundary of $L$, $f \circ \xi_1 = Sf \circ \xi_1$. Now both $f \circ \xi_t$ and $Sf \circ \xi_t$ are induced by ambient contact isotopies $\psi_t$ and $S \psi_t$, respectively, whose contact Hamiltonians vanish along the Legendrians. Thus, we may assume the energies of $\psi_t$ and $S \psi_t$ to be arbitrarily small. Therefore, if $(U,Sf(L) \cap U)$ is a small loose chart for $Sf$, then $(\psi^{-1}_1 \circ S\psi_1)(U,Sf(L) \cap U)$ is a small loose chart for $f$.\\
\end{rmk}

As the following corollary shows, we can also approximate an arbitrary Hamiltonian function instead of requiring the energy of the isotopy to be very small.

\begin{cor}\label{cor:approximations_of_arbitrary_hamiltonians}
	Assume that the assumptions of Theorem \ref{thm:main result} hold, and let $H: M \times [0,1] \to \R$ be a compactly supported Hamiltonian. Then the conclusion of Theorem \ref{thm:main result} remains true if we replace the condition $\Vert \phi_t \Vert_\alpha < \eta$ by $\Vert H_t - F_t \Vert < \eta$ , where $F_t$ denotes the contact Hamiltonian associated to $\phi_t$.
\end{cor}	

Corollary \ref{cor:approximations_of_arbitrary_hamiltonians} is proven in Section \ref{sec:proof of stabilized case}.\\

Proposition \ref{prop:jet-space_with_small_energies} shows that local Weinstein neighbourhoods of $L_0$ and $L_1$ of height $2 \eta$ satisfy the assumption on the sets $U_i$ with $\eps = 2\eta$. Furthermore, recall that stabilized Legendrians are always loose. Applying the above theorems to this case gives the following results.

\begin{cor}\label{cor:isotopy_of_stabilized_legendrian}
	Let $L_0,L_1 \subseteq (M^{2n+1},\alpha)$, $n \geq 2$, be closed Legendrian submanifolds. For $i \in \{0,1\}$, stabilize $L_i$ inside of a local Weinstein neighbourhood $U_i$ of $L_i$ of height $2 \eps_i>0$ to obtain a new Legendrian $SL_i$. Let $V_i$ be an open neighbourhood of $L_i$ such that $U_i \subseteq V_i$. If there exists a family of Legendrian embeddings $f_t:L \to M, t \in [0,1],$ with $f_i(L) = SL_i$ then, for any given $\eta>0$, there exist contact isotopies $\phi_t$, and $\psi_t^i,i \in \{0,1\}$, with $\supp(\psi^i_t) \subseteq V_i$, $\Vert \psi^i_t \Vert_\alpha < \eps_i$, $\Vert \phi_t \Vert_\alpha < \min\{\eta, (\eps_0 - \Vert \psi^0_t \Vert_\alpha),(\eps_1 - \Vert \psi^1_t \Vert_\alpha)\}$, and $\psi^1_1 \circ \phi_1 \circ \psi^0_1 \circ f_0 = f_1$. Furthermore, given any $\delta>0$, these isotopies can be chosen in such a way that $\phi_t\circ \psi^0_1 \circ f_0$ is $\delta$-close to $f_t$ for all $t \in [0,1]$. In particular, the energy of the concatenation $(\psi^0 * \phi * \psi^1)_t$ is smaller than $\eps_0 + \eps_1$.
\end{cor}

\begin{rmk}
	Again, it is possible to approximate arbitrary Hamiltonians as in Corollary \ref{cor:approximations_of_arbitrary_hamiltonians} instead of the condition on $\Vert \phi_t \Vert_\alpha$.
\end{rmk}
	
\begin{cor}\label{cor:displacement_of_stabilized_legendrian}
	Let $L^n \subseteq (M,\alpha)$, $n \geq 2$, be a closed Legendrian submanifold, and let $U_\eps$ be a local Weinstein neighbourhood of $L$ of height $2 \eps$. Let $SL$ denote a Legendrian submanifold of $M$ obtained by stabilizing $L$ inside of $U_\eps$, and assume that $SL$ is displaceable. Then for any $E > 0$, there exists a compactly supported contact isotopy $\phi_t$ with $\Vert \phi_t \Vert_\alpha < \eps$ such that all Reeb chords between $SL$ and $\phi_1(SL)$ have action larger than $E$. 
	
	If in addition the image of $SL$ under the Reeb flow is closed as a subset of $M$, we may assume that there are no Reeb chords between $SL$ and $\phi_1(SL)$. In particular, the displacement energy of $SL$ is not larger than $\eps$.
\end{cor}

\noindent Corollaries \ref{cor:isotopy_of_stabilized_legendrian} and \ref{cor:displacement_of_stabilized_legendrian} are proven in Section \ref{sec:proof of stabilized case}.

\begin{rmk}
	Note that the supremum of the actions of Reeb chords in $U_\eps$ is bounded by $2 \eps$. This means that Corollary \ref{cor:displacement_of_stabilized_legendrian} states the reverse of the energy capacity inequality in Theorem \ref{thm:leg energy capacity ineq} for stabilized Legendrians.\\
\end{rmk}

\begin{rmk}\label{rmk:size of loose chart}
	The proof of Proposition \ref{prop:jet-space_with_small_energies} motivates the following coordinate independent definition of the size of a loose chart. Namely, let $(M, \alpha)$ be a strict contact manifold, $L \subseteq M$ a Legendrian, and $U \subseteq M$ a connected open set so that $U \cap L \subseteq U$ is loose. Then we say that $L$ has a loose chart of size 
	\begin{equation}
	2 \sup\limits_{V} \inf\limits_{\phi_t} \Vert \phi_t \Vert_\alpha
	\end{equation}
	in $U$, where the supremum is taken over all open sets $V \subseteq U$, and the infimum is taken over all contact isotopies $\phi_t$ with support in $U$ such that $(V,\phi_1(L) \cap V)$ contains a loose chart. In other words, the size of a loose chart is (up to a factor of $2$) the minimal energy that is required to produce an arbitrarily small loose chart. This follows from the observations that (1) if $V' \subseteq V$ and $(V',\phi_1(L) \cap V')$ contains a loose chart then also $(V,\phi_1(L) \cap V)$ contains a loose chart, and that (2) the position of a small open set $V \subseteq U$ does not matter as any two small open sets can be moved into each other via a contact isotopy that has small energy.
	
	In particular, if the loose chart is contained in a local Weinstein neighbourhood of height $\eps$, then its size will be smaller than $\eps$ by the arguments in the proof of Proposition \ref{prop:jet-space_with_small_energies}.
	
	Note that the size of a loose chart depends on the chosen contact form. For example, if we replace $\alpha$ by $\lambda \alpha$ for some $\lambda > 0$, then it follows from the definition of the contact Hamiltonian associated to a contact isotopy that the size of any loose chart changes by a factor of $\lambda$ as well.
	
	With this definition, the above results can be restated vaguely as ``If $L_0$ and $L_1$ have loose charts of size $\eps_0$ and $\eps_1$, respectively, then, for any $\eta > 0$, one can approximate any Legendrian isotopy from $L_0$ to $L_1$ by a Legendrian isotopy of energy less than $\frac{\eps_0}{2} + \frac{\eps_1}{2} + \eta$". Combining this with Murphy's h-principle for loose Legendrians yields the following quantitative version of the h-principle.
	
	\begin{cor}\label{cor:quantitative h-principle}
		Let $L_0$ and $L_1$ be two closed loose Legendrians that are formally isotopic and admit loose charts of size $\eps_0$ and $\eps_1$ (in $M$), respectively. Then, for any $\eta > 0$, there exist a Legendrian isotopy from $L_0$ to  $L_1$  of energy less than $\frac{\eps_0}{2} + \frac{\eps_1}{2} + \eta$.
	\end{cor}

	Corollary \ref{cor:quantitative h-principle} is proven in Section \ref{sec:proof of stabilized case}.\\
\end{rmk}

\textbf{Acknowledgments:} The idea for this paper originated in discussions with Georgios Dimitroglou Rizell. I am very grateful to him for many helpful comments and for suggesting the definition of the size of a loose chart in Remark \ref{rmk:size of loose chart} (personal communication, December 3, 2020). I am also grateful to the referee for pointing out gaps in the proofs of the main result and Theorem 1.4 in the original version of this paper, and for many comments that helped to greatly improve this article. Also, I would like to thank my adviser Tobias Ekholm for interesting discussions about the contents of this paper.

This work is partially supported by the Knut and Alice Wallenberg Foundation, grant KAW2020.0307.\\

\section{The energy of a contactomorphism}\label{sec:the_energy_of_a_contactomorphism}

Let $(M, \alpha)$ be a strict contact manifold, and let $\Cont(M)$ denote the set of all compactly supported contactomorphisms on $M$ that are isotopic to the identity through compactly supported contactomorphisms. For any function $H:  M \times [0,1]  \to \R$ with compact support define
\begin{equation}
\Vert H \Vert \coloneqq \int_0^1 \underset{x \in M}{\text{max}}\, |H(x,s)| ds.
\end{equation}
$\Vert \phi_t^H \Vert_\alpha \coloneqq \Vert H \Vert$ will also be called the  \emph{energy} of $\phi^H_t$, where $\phi^H_t$ denotes the contact isotopy associated to $H$.

For any contactomorphism $\phi \in \Cont(M)$ define\footnote{By abuse of notation, we will use the same symbol to denote the energy of a contact isotopy and a contactomorphism. When the contactomorphism is written with the subscript $t$, then we will always mean the energy of the contact isotopy.}
\begin{equation}
\Vert \phi \Vert_\alpha \coloneqq  \underset{\psi_t}{\inf}\, \Vert \psi_t \Vert_\alpha,
\end{equation}
where the infimum is taken over all compactly supported contact isotopies $\psi_t$ on $M$ that satisfy $\psi_1 = \phi$. Recall that in this paper all isotopies start at the identity.

Shelukhin proved the following theorem.

\begin{thm}\label{thm:contact hofer norm}\hspace{-1mm}\textnormal{\cite{she17}}\, Let $\phi, \psi \in \Cont(M)$ be two contactomorphisms. Then
	
	(i) (non-degeneracy) $\Vert \phi \Vert_\alpha \geq 0$, and $\Vert \phi \Vert_\alpha = 0$ if and only if $\phi = id_M$.
	
	(ii) (triangle inequality) $\Vert \phi \psi \Vert_\alpha \leq \Vert \phi \Vert_\alpha + \Vert \psi \Vert_\alpha$.
	
	(iii) (symmetry) $\Vert \phi^{-1} \Vert_\alpha = \Vert \phi \Vert_\alpha$.
	
	(iv) (naturality) $\Vert \psi \phi \psi^{-1} \Vert_\alpha = \Vert \phi \Vert_{\psi^*\alpha}$.\\
\end{thm}

In fact, the following lemma is an easy consequence of the definition of a contact Hamiltonian isotopy associated to a contact isotopy.

\begin{lem}\label{lem:hamiltonian_under_conjugation}
	Let $H:M \times [0,1] \to \R$ be a compactly supported Hamiltonian with associated contact isotopy $\phi^H_t$. Then $-H_{1-t}$ is the contact Hamiltonian associated to $\phi^H_{1-t} \circ (\phi^H_1)^{-1}$.
	
	Let $\psi:M \to M$ be a contactomorphism and denote by $f:M \to \R_{>0}$ the function defined by $\psi^*\alpha = f \alpha$. Then $(f H_t)\circ \psi^{-1}$ is the Hamiltonian associated to the contact isotopy $\psi \phi^H_t \psi^{-1}$.\\
\end{lem}

Rosen and Zhang \cite{rz18} analyzed how $\Vert \cdot \Vert_\alpha$ behaves with respect to orbits of certain subsets of $M$ under the contactomorphism group. The following result follows immediately from Theorem 1.10 and Proposition 8.6 in \cite{rz18} and will be essential in our argument.

\begin{thm}\label{thm:hofer_distance_for_non-legednrians}\hspace{-1mm}\textnormal{\cite{rz18}}\,
	Let $L^n \subseteq M^{2n+1}$ be a closed, connected non-Legendrian submanifold, and let $\Phi_t:M \to M$ be a compactly supported contact isotopy. Then there exist compactly supported contact isotopies $\phi_t:M \to M$ of arbitrarily small energies that satisfy $\phi_1(L)=\Phi_1(L)$.
\end{thm}

Recently, Dimitroglou Rizell and Sullivan \cite{drs22} refined Rosen and Zhang's result to (parametrized) proper non-Legendrian embeddings. The following is a weaker version of Theorem B in \cite{drs22}.

\begin{thm}\label{thm:hofer distance for parametrized non-legendrians}
	Let  $f:L^n \to M^{2n+1}$ be a proper, connected non-Legendrian embedding of a manifold $L$, and let $\Phi_t:M \to M$ be a compactly supported contact isotopy. Then there exist compactly supported contact isotopies $\phi_t:M \to M$ of arbitrarily small energies that satisfy $\phi_1 \circ f = \Phi_1 \circ f$.
\end{thm}

\begin{rmk}
	From the proof of Theorem B in \cite{drs22} it is clear that Theorem \ref{thm:hofer distance for parametrized non-legendrians} also holds in the case that $L$ is disconnected with finitely many connected components as long as we assume that every component is non-Legendrian since the heart of the argument is purely local around the image of any connected component of $L$. Furthermore, even if $L$ has countably infinitely many connected components, only finitely many components of $f(L)$ can intersect the (compact!) support of $\Phi_t$ by properness of $f$. Thus, the conclusion of Theorem \ref{thm:hofer distance for parametrized non-legendrians} also holds in this case (still under the assumption that f is non-Legendrian on every connected component of $L$).
\end{rmk}

We will need the following $C^0$-close version of Theorem \ref{thm:hofer distance for parametrized non-legendrians}.

\begin{cor}\label{cor:c^0_close_hofer_distance_for_non-legendrians}
	Let  $f:L^n \to M^{2n+1}$ be a proper embedding which is non-Legendrian almost everywhere (i.e. $D_x f (T_x L) \nsubseteq \ker \alpha$ a.e. $x \in L$), and let $\Phi_t:M \to M$ be a compactly supported contact isotopy. Then for any $\delta > 0$, there exist compactly supported contact isotopies $\phi_t:M \to M$ of arbitrarily small energies that satisfy $\phi_1 \circ f = \Phi_1 \circ f$ so that $\phi_t \circ f$ is $\delta$-close to $\Phi_t \circ f$ for all $t \in [0,1]$.
\end{cor}

\bproof Let $\{U^k_{j}\}, k \in \{1,...,K_j\}, j \in \{1,...,J\},$ be an open cover of the support of $\Phi_t$ by relatively compact subsets such that 

(1) the diameter of $U_j^k$ is smaller than $\frac{\delta}{J+1}$ for all $j,k$,

(2) $U_j^k \cap U_j^{k'} = \emptyset$ for all $j \in \{1,...,J\}, k,k' \in \{1,...,K_j\}, k \neq k'$.\\

The existence of such sets can be seen as follows. After choosing a proper embedding of $M$ into some $\R^H,H \in \N,$ we may assume that $M$ is a proper submanifold of $\R^H$. For any $\lambda > 0$, $R^H$ is covered by the cubes 
\begin{equation}
	C^k_{\lambda,j} \coloneqq \Pi_{i=1}^H (\lambda(2k_i + j_i - \frac{1}{3}), \lambda(2k_i + j_i + \frac{4}{3})),
\end{equation}
where $k = (k_1,...,k_H) \in \Z^H$ and $j = (j_1,...,j_H) \in \{0,1\}^H$. We set $U_{\lambda,j}^k \coloneqq C_{\lambda,j}^k \cap M$. Then (2) is clearly satisfied. We define 
\begin{equation}
\mathcal{K}_\lambda \coloneqq \{k \in \Z^H| \exists j \in \{0,1\}^H : U^k_{\lambda,j} \cap \supp\, \Phi_t \neq \emptyset\}.
\end{equation}

Then the sets $\{U_{\lambda,j}^k\}, j \in \{0,1\}^H, k \in \mathcal{K}_{\lambda},$ cover the support of $\Phi_t$, and it is straightforward to see that (1) is satisfied if $\lambda$ is sufficiently small.\\

By the proof of the Fragmentation Lemma in \cite{ban97}, there exist a subdivision $t_0 = 0 < t_1 <...<t_N=1, N \in \N$, of the interval $[0,1]$ and contact isotopies $\phi^{i,j}_t, j \in \{1,...,J\}, t \in [t_i,t_{i+1}], i \in \{0,...,N-1\},$ such that each $\phi^{i,j}_t$ is supported in $\bigcup_{1 \leq k \leq K_j} U^k_j$, and for every $i \in \{0,...,N-1\}$, $\Phi_t \circ (\Phi_{t_i})^{-1} = \phi^{i,J}_t \circ ... \circ \phi^{i,2}_t \circ \phi^{i,1}_t, \quad t \in [t_i,t_{i+1}]$. Furthermore, we may assume that $\Phi_t \circ (\Phi_{t_i})^{-1}, t \in [t_i, t_{i+1}],$ is $ \frac{\delta}{J+1}$-close to the identity for every $i \in \{0,...,N-1\}$.\\

First, we assume that $N = 1$, and we drop $i$ from the notation. Let $\eps > 0$. We will define contact isotopies $\widetilde{\phi}^j_t$ inductively over $j$. Since $f$ is non-Legendrian almost everywhere, $f|_{f^{-1}(\bigcup_{1 \leq k \leq K_1} U^k_1 \cap f(L))}$ is non-Legendrian on each connected component of $f^{-1}(\bigcup_{1 \leq k \leq K_1} U^k_1 \cap f(L))$. By Theorem \ref{thm:hofer distance for parametrized non-legendrians} applied to $\phi^1_t|_{\bigcup_{1 \leq k \leq K_1} U^k_1}$ and $f|_{f^{-1}(\bigcup_{1 \leq k \leq K_1} U^k_1 \cap f(L))}$, there exists a compactly supported contact isotopy $\widetilde{\phi}^1_t$ on $\bigcup_{1 \leq k \leq K_1} U^k_1$ such that $\Vert \widetilde{\phi}^1_t \Vert_\alpha < \frac{\eps}{J}$ and $\widetilde{\phi}^1_1 \circ f = \phi^1_1 \circ f$. In the following we will view $\widetilde{\phi}^1_t$ as a contact isotopy of $M$ with compact support in $\bigcup_{1 \leq k \leq K_1} U^k_1$. Assume now that $\widetilde{\phi}^1_t,...,\widetilde{\phi}^j_t$ have already been defined. As before, it follows from Theorem \ref{thm:hofer distance for parametrized non-legendrians} applied to $\phi^{j+1}_t|_{\bigcup_{1 \leq k \leq K_{j+1}} U^k_{j+1}}$ and $\phi^j_1 \circ ... \circ \phi^1_1 \circ f|_{(\phi^j_1 \circ ... \circ \phi^1_1 \circ f)^{-1}(\bigcup_{1 \leq k \leq K_{j+1}} U^k_{j+1} \cap (\phi^j_1 \circ ... \circ \phi^1_1 \circ f)(L))}$ that we can find a contact isotopy $\widetilde{\phi}^{j+1}_t$ with compact support in $\bigcup_{1 \leq k \leq K_{j+1}} U^k_{j+1}$ such that $\Vert \widetilde{\phi}^{j+1}_t \Vert_\alpha < \frac{\eps}{J}$ and $\widetilde{\phi}^{j+1}_1 \circ \phi^j_1 \circ ... \circ \phi^1_1 \circ f = \phi^{j+1}_1 \circ \phi^j_1 \circ ... \circ \phi^1_1 \circ f$. After $J$ steps we have defined isotopies $\widetilde{\phi}^1_t,...,\widetilde{\phi}^{J}_t$ such that $\phi_t \coloneqq \widetilde{\phi}^{1}_t * ...* \widetilde{\phi}^{J}_t$ satisfies $\phi_1 \circ f = \Phi_1 \circ f$, where $*$ denotes the concatenation of isotopies. 

Furthermore, $\Vert \phi_t \Vert_\alpha < \eps$ and each $\widetilde{\phi}^j_t$ is $\frac{\delta}{J+1}$-close to the identity since it is supported in a disjoint union of sets of diameter less than $\frac{\delta}{J+1}$. In particular, $\phi_t$ is $\frac{J\delta}{J+1}$-close to the identity. Also, $\Phi_t$ is $\frac{\delta}{J+1}$-close to the identity by assumption. All in all, we see that $\phi_t \circ f$ is $\delta$-close to $\Phi_t \circ f$ for all $t \in [0,1]$. This finishes the proof for $N=1$. 

In the case that $N > 1$, we perform the above constructions on each time interval $[t_i, t_{i+1}]$ separately (with $\eps$ replaced by $\frac{\eps}{N}$) and concatenate the obtained contact isotopies to find the desired isotopy $\phi_t$.
\eproof\\

\begin{rmk}\label{rmk:comparison_oscillatory_and_maximums_norm}
	Similarly to the Hofer metric in symplectic manifolds, one can also consider the oscillatory semi-norm 
	\begin{equation}
	\Vert H \Vert_{osc} \coloneqq \int_0^1 \left(\underset{x \in M}{\max}\, H(x,s) - \underset{x \in M}{\min}\, H(x,s)\right) ds
	\end{equation}
	of a compactly supported function $H:M \times [0,1] \to \R$. If $M$ is non-compact, 
	\begin{equation}
	2\,\underset{x \in M}{\max}\, |H(x,t)| \geq \underset{x \in M}{\max}\, H(x,t) - \underset{x \in M}{\min}\, H(x,t) \geq \underset{x \in M}{\max}\, |H(x,t)|
	\end{equation}
	holds for all times $t \in [0,1]$. This implies that the same inequalities also hold for the energies of contact isotopies that are defined using $\Vert \cdot \Vert$ and $\Vert \cdot \Vert_{osc}$. 
	
	In fact, when dealing with displacement, these two semi-norms differ exactly by a factor of $2$ in the following sense. Assume that $A_0, A_1 \subseteq M$ are two compact subsets such that there exists a contact isotopy $\phi_t:M \to M$ with $\phi_1(A_0)=A_1$. Assume that the Reeb vector field is complete and denote the Reeb flow by $\phi_t^\alpha$. Then a straightforward argument shows that 
	\begin{equation}
	2 \underset{T \in \R}{\inf} \underset{\phi_t}{\inf} \Vert \phi_t \Vert_\alpha = \underset{\phi_t}{\inf} \Vert \phi_t \Vert_{osc}
	\end{equation}
	where the infimum on the left (resp. right) hand side is taken over all compactly supported contact isotopies $\phi_t:M \to M$ with $\phi^H_1(A_0) = \phi^\alpha_T(A_1)$ (resp. $\phi^H_1(A_0) = A_1$).\\
\end{rmk}

Let $N$ be a manifold and denote its 1-jet bundle by $J^1 N$. For $\delta > 0$ we define \emph{a local Weinstein neighbourhood of the zero section of height $2 \delta$} to be an open set of the form 
\begin{equation}\label{def:local_weinstein_neighbourhood}
	U_\delta \coloneqq (\{(q,p,z)\in J^1 N| q \in V, p \in W_q, |z|< \delta_q \}, \alpha_{std}),
\end{equation}
where $V \subseteq N$ is open, $W_q \subseteq T^*_q N$ is a star-shaped neighbourhood of $0$, and $V \to (0,\infty), q \mapsto \delta_q,$ is a function with $\underset{q \in V}{\sup} \,\delta_q \leq \delta$. Here, $\alpha_{std}$ is locally defined as $\alpha_{std} = \sum_i p_i d q_i -dz$, where $\{q_i\}$ are local coordinates on $N$, and $\{p_i\}$ are the conjugate coordinates on the cotangent fibers. 

Similarly, for a Legendrian submanifold $N \subseteq (M,\alpha)$ we call $U_\delta \subseteq M$ \emph{a local Weinstein neighbourhood of $N$ of height $2 \delta$} if it is strictly contactomorphic to a local Weinstein neighbourhood $\widetilde{U}_\delta \subseteq J^1 N$ of the zero section of height $2 \delta$ via a contactomorphism that identifies $N \cap U_\delta$ with $N \cap \widetilde{U}_\delta$. Recall that any closed Legendrian has strict Weinstein neighbourhoods (\cite{gei08}, Theorem 6.2.2). Such a neighbourhood $U_\delta \approx \widetilde{U}_{\delta}$ is said to have fibers of diameter $\eps > 0$ (with respect to some metric on $M$) if the diameter in $M$ of the set of points identified with $\pi^{-1}(\{q\}) \subseteq \widetilde{U}_{\delta}$ is less than $\eps$ for all $q \in N \cap \widetilde{U}_\delta \approx N \cap U_\delta$, where $\pi: J^1N \to N$  denotes the projection onto the zero section.

The following proposition gives us bounds on the energies of contact isotopies with support in local Weinstein neighbourhoods. The first part is taken from \cite{drs20}.

\begin{prop}\label{prop:jet-space_with_small_energies}
	For any manifold $N$ and any local Weinstein neighbourhood $U_\delta$ of the zero section of height $2 \delta > 0 $ and every compactly supported contact isotopy $\phi_t:U_\delta \to U_\delta$, $t \in [0,1]$, there exists a compactly supported contact isotopy $\widetilde{\phi}_t: U_\delta \to U_\delta$ such that $\widetilde{\phi}_1 = \phi_1$ and $\Vert \widetilde{\phi}_t \Vert_\alpha < 2 \delta$.
	
	If, in addition, $U_\delta$ has fibers of diameter $\eps > 0$, then for any compact subset $K \subseteq U_\delta$, $\widetilde{\phi}_t|_K$ can be chosen to be $\eps$-close to $\phi_t|_K$ for all $t \in [0,1]$.
\end{prop}

\bproof
	We assume that $N$ is closed and that $U_\delta = 	J^1_\delta N \coloneqq \{(q,p,z) \in J^1 N| \left|z\right| < \delta \}$. The general case is directly analogous by restricting to a compact subset containing the support of $\phi_t$. Let $0< \eta < \delta$ be such that the isotopy $\phi_t$ is supported in $J^1_{\eta}N$. For $\lambda >0$ consider a time-dependent function $H_t: M \to \R$, $t \in [0,1]$, such that $H_t(q,p,z) = \lambda z$ on $J^1_{e^{- \lambda t} \eta}N$, and such that $H_t$ is appropriately cut-off outside of $J^1_{e^{-\lambda t}\eta}N$. Then its associated contact isotopy $\psi^H_t$ will map $(q,p,z) \in J^1_{\eta}N$ to $(q,e^{-\lambda t} p, e^{-\lambda t}z) \in J^1_{e^{-\lambda t}\eta}N$, and it satisfies
	\begin{equation}
	\Vert \psi^H_t \Vert_\alpha \leq \int\limits_0^1 \lambda \eta e^{- \lambda t} dt + \eta e^{- \lambda} = \eta,
	\end{equation}
	where the $\eta e^{- \lambda}$-summand is due to the chosen cut-off. Note also that $((\psi^H_1)^* \alpha)|_{J^1_{\eta} N} = e^{- \lambda} \alpha|_{J^1_{\eta} N}$. 
	
	Now let 
	\begin{equation}
		\widetilde{\phi}_t \eqqcolon \left((\psi^H_s) * (\psi^H_1 \phi_s (\psi^H_1)^{-1}) * (\psi^H_{1-s} \circ (\psi^H_1)^{-1}) \right)_t,
	\end{equation}
	where $*$ denotes the concatenation of isotopies. 

	As a consequence of Lemma \ref{lem:hamiltonian_under_conjugation}, 
	\begin{equation}
	\begin{split}
		\Vert \widetilde{\phi}_t \Vert_\alpha &= \Vert \left( (\psi^H_s) *  (\psi^H_1 \phi_s (\psi^H_1)^{-1}) * (\psi^H_{1-s} \circ (\psi^H_1)^{-1}) \right)_t \Vert_\alpha \\
		&= \Vert \psi^H_t \Vert_\alpha + \Vert \psi^H_1 \phi_t (\psi^H_1)^{-1} \Vert_\alpha  + \Vert \psi^H_{1-t} \circ (\psi^H_1)^{-1} \Vert_\alpha\\
		&= 2 \Vert \psi^H_t \Vert_\alpha + e^{-\lambda} \Vert \phi_t \Vert_\alpha \leq 2 \eta + e^{-\lambda} \Vert \phi_t \Vert_\alpha,
	\end{split}
	\end{equation}
	which is smaller than $2 \delta$ if $\lambda$ is sufficiently large.
	
	Now assume that $K \subseteq U_\delta$ is a compact subset. By choosing $\eta$ sufficiently close to $\delta$, we may assume that $K \subseteq J^1_\eta N$. If, in addition, $U_\delta$ has fibers of diameter $\eps > 0$, then $\psi_t^H|_{J^1_\eta N}$ and $\psi^H_{1-t} \circ (\psi^H_1)^{-1}|_{J^1_{e^\lambda \eta} N}$ are $\eps$-close to the identity as they preserves the fibers of the projection onto the zero section. Therefore, $\psi^H_1 \phi_t (\psi^H_1)^{-1} \circ \psi^H_1|_{J^1_\eta N} = \psi^H_1 \phi_t|_{J^1_\eta N}$ is $\eps$-close to $\phi_t|_{J^1_\eta N}$, and $\widetilde{\phi}_t|_{J^1_\eta N}$ will be $\eps$-close to $\phi_t|_{J^1_\eta N}$ if we perform the concatenation in the definition of $\widetilde{\phi}_t$ is such a way that $\psi_t^H$ and $\psi^H_{1-t} \circ (\psi^H_1)^{-1}$ are traversed very quickly.
\eproof\\

\section{Loose Legendrians}\label{sec:loose legendrians}

Murphy's h-principle for loose Legendrians is an important ingredient in the proof of our main result. Roughly speaking, it states that for two loose Legendrians in the sense of \cite{mur12} the existence of a Legendrian isotopy between them is a purely homotopy-theoretical problem. We recall this result in this section and explain how to refine Murphy's proof to obtain $C^0$-bounds. 

Recall that a formal Legendrian embedding of an $n$-dimensional manifold $L$ into a contact manifold $(M^{2n+1},\xi)$ is a pair $(f,F_s)$ consisting of an embedding $f:L \to M$ and a homotopy of fibrewise injective bundle maps $F_s: TL \to f^*(TM), s\in [0,1],$ covering $f$ such that $F_0$ is equal to the differential $Df$ of $f$ and $F_1(TL)$ is a Lagrangian subspace of $f^*\xi$ at every point with respect to the conformal symplectic structure on $\xi$. We identify a Legendrian embedding $f: L \to M$ with the formal Legendrian embedding $(f,Df)$ where $Df$ is viewed as the constant homotopy.

The first result that we need is the following version of Murphy's h-principle. 

\begin{thm}\label{thm:murphys h-principle}
	Let $(M^{2n+1},\xi)$, $n \geq 2$, be a contact manifold endowed with a Riemannian metric, and let $L^n$ be a connected manifold. Let $(f_t, F_s^t), t \in [0,1]$ be a homotopy of proper formal Legendrian embeddings $L \to (M, \xi)$, constant (in $t$) outside of a compact subset of $L$, such that $(f_0,F_s^0)$ and $(f_1,F_s^1)$ are Legendrian embeddings which admit loose charts $(U_0, U_0 \cap f_0(L))$ and $(U_1,U_1 \cap f_1(L))$, respectively. Then for any $\delta > 0$, there exists a homotopy $\Phi_t: L \to M$ of Legendrian embeddings, constant outside of a compact subset of $L$, such that $\Phi_i = f_i$ for $i \in \{0,1\}$, and $\Phi_t$ is pointwise $(d + \delta)$-close to $f_t$ for all $t \in [0,1]$, where $d$ denotes the maximum of the diameters of $U_0$ and $U_1$ (with distances measured in $M$).
\end{thm}

\bproof
We will explain how to adjust the arguments in the proof of Theorem 1.2 in \cite{mur12} to prove this theorem. The $C^0$-close part in the case of a fixed loose chart and the fact that we may choose compactly supported homotopies are consequences of the constructions in Murphy's proof. We will first explain how to obtain these two results and then reduce the proof in the general case to these special cases.

Let $\delta > 0$ be an arbitrary number, and define $d$ as in the statement.

First note that the assumptions of the theorem imply in particular that $f_t$ is Legendrian outside of a compact subset of $L$ for all $t$. Furthermore, we may assume that $f_0^{-1} (U_0 \cap f_0(L))$ and $f_1^{-1} (U_1 \cap f_1(L))$ are contained in a compact subset of $L$ after possibly replacing $U_0$ and $U_1$ by slightly smaller loose charts. Note that this will not increase $d$.

First, we prove the theorem under the assumptions that $U_0 = U_1$, and that $f_t^{-1} (U_0 \cap f_0(L))$ and $f_t|_{f_t^{-1} (U_0 \cap f_0(L))}$ do not depend on $t$ (in other words, we assume that there is a fixed loose chart for the family $f_t$).

By Proposition 3.4 and the subsequent paragraph in \cite{mur12}, there exists a compactly supported homotopy $\bar{f}_t:L \to M$ of wrinkled Legendrian embeddings\footnote{Here and below, we omit the data of the Darboux charts in the definition of a wrinkled Legendrian from the notation.} agreeing with $f_t$ outside of a compact subset such that $\bar{f}_t^{-1} (U_0 \cap f_0(L))$ and $\bar{f}_t|_{\bar{f}_t^{-1} (U_0 \cap f_0(L))}$ do not depend on $t$, $\bar{f}_i = f_i$ for $i \in \{0,1\}$, and $\bar{f}_t$ is $\frac{\delta}{2}$-close to $f_t$.

Following \cite{mur12}, there exists a homotopy $g_t$ of wrinkled Legendrian embeddings obtained from $\bar{f}_t$ by replacing loose charts in $U_0$ by inside-out wrinkles which agrees with $\bar{f}_t$ outside of $f_0^{-1} (U_0 \cap f_0(L))$, satisfies $g_t(f_0^{-1} (U_0 \cap f_0(L))) \subseteq U_0$, and admits a (finite) collection of markings for its wrinkles which agree with the markings of the model inside-out wrinkle near $t \in \{0,1\}$. It follows that $g_t$ is $d$-close to $\bar{f}_t$ since $U_0$ has diameter $d$. 

Let $\widetilde{g}_t: L \to M$ denote the homotopy of (smooth) Legendrian embeddings obtained from $g_t$ by resolving the singularities using the collection of markings (see Lemma 4.2 in \cite{mur12}). Note that $\widetilde{g}_t$ agrees with $f_t$ outside of a compact subset of $L$. Furthermore, we may assume that $\widetilde{g}_t$ is $\frac{\delta}{2}$-close to $g_t$.

Combining the $C^0$- estimates, we see that $\widetilde{g}_t$ is $(d+ \delta)$-close to $f_t$.

The explicit construction of an inside-out wrinkle in \cite{mur12} shows that, in addition, we may assume that for $i \in \{0,1\}$, there exists a homotopy $h^i_t:L \to M$ of Legendrian embeddings which agrees with $f_i = \bar{f}_i$ outside of $f_0^{-1} (U_0 \cap f_0(L))$ so that $h^i_0 = f_i$, $h^i_1 = \widetilde{g}_i$, and $h^i_t(f_0^{-1} (U_0 \cap f_0(L))) \subseteq U_0$. It follows again that $h^i_t$ is $d$-close to $f_i$.

Now consider the homotopy $\Phi_t = \left( h^0 * \widetilde{g} * \overline{h^1} \right)_t$ of Legendrian embeddings, where $\overline{h^1_t} = h^1_{1-t}$. By construction, $\Phi_i = f_i$ for $i \in \{0,1\}$, and $\Phi_t$ agrees with $f_t$ outside of a compact subset of $L$. Since $h^i_t$, $f_t$ and $\widetilde{g}_t$ are constant (in $t$) outside of a compact set (in fact, they are equal to each other outside of a compact set), $h^i_t$ is $d$-close to $f_i$ for $i \in \{0,1\}$, and $\widetilde{g}_t$ is $(d + \delta)$-close to $f_t$, $\Phi_t$ will be $(d + \delta)$-close to $f_t$ for all $t \in [0,1]$ if the concatenation in the definition of $\Phi_t$ is performed in such a way that $h^0_t$ and $\overline{h^1_t}$ are traversed sufficiently fast.

It follows that $\Phi_t$ satisfies the required properties.

This finishes the proof in the case that there exists a fixed loose chart.\\

In the general case, pick a path $p_t, t \in [0,1],$ in $L$ so that $f_i$ has a loose chart of diameter bounded by $d$ in the complement of some neighbourhood of $f_i(p_i)$ for $i \in \{0,1\}$. We may assume that $(f_t,F^t_s)$ is Legendrian in a neighbourhood of $p_t$ for all $t \in [0,1]$ by Lemma \ref{lem:formal leg leg near a point} below. Let $\xi_t$ be a compactly supported isotopy of $L$ with $\xi_t(p) = p_t$ for all $t$. After replacing $(f_t, F_s^t)$ by $(f_t \circ \xi_t,F_s^t \circ D \xi_t)$, we may assume that $p \coloneqq p_t$ does not depend on $t$. Then there exists a compactly supported contact isotopy $\psi_t:M \to M$ such that $\psi_t \circ f_0 = f_t$ on a neighbourhood of $p$ since homotopies of compact Legendrian embeddings can always be extended to contact isotopies. Because $\psi_t$ has compact support, it is uniformly $C^0$-continuous. Thus, we can find $\eps>0$ with $\eps < \frac{\delta}{2}$ such that for any two points $x,y \in M$, the distance between $\psi_t(x)$ and $\psi_t(y)$ is smaller than $\frac{\delta}{2}$ for all $t \in [0,1]$ whenever the distance between $x$ and $y$ is smaller than $\eps$.

The homotopy $(\psi_t^{-1} \circ f_t, (D \psi_t)^{-1} \circ F_s^t)$ of formal Legendrian embeddings is genuinely Legendrian for $t \in \{0,1\}$ and equal to $(f_0,Df_0)$ on a neighbourhood $V$ of $p$ for all $t \in [0,1]$, and does not depend on $t$ outside of a compact subset of $L$. Let $U$ be a Darboux ball around $f_0(p)$ so that $f_0^{-1}(U) \subseteq V$, $(\psi_t^{-1} \circ f_t)^{-1}(U) = f_0^{-1}(U)$, and $(U, U \cap f_0(L),f_0(p)) \subseteq (M,f_0(L),f_0(p))$ is contactomorphic to $(B_\rho, B_\rho \cap \R^n, \{0\}) \subseteq (\R^{2n+1},\R^n, \{0\}), \rho > 0,$ with its standard contact structure. Furthermore, we ask that the diameter of $U$ is smaller than $\frac{\eps}{2}$ and that $f_i$ has a loose chart of diameter bounded by $d$ in the complement of $\psi_i(U)$. Now let $(g_t,G_s^t)$ be a homotopy of formal Legendrian embeddings which agrees with $(\psi_t^{-1} \circ f_t, (D \psi_t)^{-1} \circ F_s^t)$ outside of a compact subset of $f_0^{-1}(U)$ so that $(g_t,G_s^t)|_{f_0^{-1}(U)}$ is Legendrian, has image contained in $U$, and does not depend on $t$, and so that there is a loose chart for $g_t(L)$ contained in a compact subset of $U$. In addition, we assume that there exists a homotopy of formal Legendrian embeddings from $(g_t,G_s^t)|_{f_0^{-1}(U)}$ to $f_0|_{f_0^{-1}(U)}$ with compact support in $f_0^{-1}(U)$ and image in $U$. For example, we can find such a $(g_t,G_s^t)$ by performing a $(\chi=0)$-stabilization of $\psi_t^{-1} \circ f_t$ inside of $U$.

In particular, $(g_t,G_s^t)$ has a fixed loose chart of diameter smaller than $\frac{\eps}{2}$, does not depend on $t$ outside of a compact subset of $L$, and is $\frac{\eps}{2}$-close to $\psi_t^{-1} \circ f_t$ for $t \in \{0,1\}$. By what we have proven above, we may find a homotopy $h_t:L \to M$ of Legendrian embeddings which does not depend on $t$ outside of a compact subset of $L$ so that $h_i = g_i$ for $i \in \{0,1\}$ and so that $h_t$ is $\frac{\eps}{2}$-close to $g_t$ for all $t \in [0,1]$. In particular, $h_t$ is $\eps$-close to $\psi_t^{-1} \circ f_t$ for all $t$. 

By our choice of $\eps$, $\psi_t \circ h_t$ is $\frac{\delta}{2}$-close to $f_t$ for all $t$, by our choice of $U$, $\psi_i \circ h_i$ has a loose chart of diameter bounded by $d$ in the complement of $\psi_i(U)$, and by our choice of $g_t$, there exists a formal Legendrian isotopy from $\psi_i \circ h_i$ to $f_i$ with support in $\psi_i(U)$ for $i \in \{0,1\}$. It follows that this formal isotopy is $\frac{\delta}{2}$-close to the constant isotopy since the diameter of $\psi_i(U)$ is bounded by $\frac{\delta}{2}$.

Now we can apply the special case of Theorem \ref{thm:murphys h-principle} which we have already proven to these formal isotopies to find for $i \in \{0,1\}$ a homotopy $\phi^i_t:L \to M$ of Legendrian embeddings from $\psi_i \circ h_i$ to $f_i$ which is $(d + \delta)$-close to $f_i$ for all $t$.

Then $\Phi_t \coloneqq \left( \phi^0_{1-s} * (\psi_s \circ h_s) * \phi^1_s \right)_t$ is a compactly supported homotopy of Legendrian embeddings from $f_0$ to $f_1$. If we perform the concatenation again in such a way that $\overline{\phi^0_s}$ and $\phi^1_s$ are traversed very quickly, then $\Phi_t$ is $(d + \delta)$-close to $f_t$ for all $t \in [0,1]$.
\eproof\\

The following lemma is needed to reduce Theorem \ref{thm:murphys h-principle} to the case of a fixed loose chart.

\begin{lem}\label{lem:formal leg leg near a point}
	Let $(f_t:L \to (M,\xi),F^t_s), t \in[0,1],$ be a homotopy of formal Legendrian embeddings so that $(f_i,F^i_s)$ is Legendrian for $i \in \{0,1\}$. Let $p_t \in L, t \in [0,1],$ be a smooth path, and let $W  \subseteq L \times [0,1]$ be an open neighbourhood of $\bigcup_t \{(p_t,t)\}$. Then there exists a $C^0$-small formal homotopy from $(f_t,F^t_s)$ to a homotopy of formal Legendrian embeddings $(g_t,G^t_s)$ which is fixed for $t \in \{0,1\}$ and outside of a compact subset of $W$, so that $(g_t, G^t_s)$ is genuinely Legendrian on a neighbourhood of $p_t$ for all $t$.
\end{lem}

\bproof  \emph{Step 1. We may assume that $p_t$ does not depend on $t$.}

Let $\xi_t$ be a compactly supported isotopy of $L$ so that $\xi_t(p) = p_t$ for all $t$. If the lemma holds for $(f_t \circ \xi_t, F^t_s \circ D \xi_t)$, then it also holds for $(f_t, F^t_s)$.\\

\emph{Step 2. We may assume that $f_t(p)$ does not depend on $t$.}

Let $\Phi_t$ be a compactly supported contact isotopy of $M$ so that $\Phi_t(f_0(p)) = f_t(p)$. If the lemma holds for $(\Phi_t^{-1}\circ f_t, (D\Phi_t)^{-1} \circ F^t_s)$, then it also holds for $(f_t, F^t_s)$ by uniform $C^0$-continuity of $\Phi_t$.\\

As all of the constructions below are inherently performed inside of an arbitrarily small neighbourhood of $p$, the homotopies of formal Legendrian homotopies will be supported in a compact subset of $W$.\\

\emph{Step 3: We may assume that $F^t_s(p) = D_p f_t$ for all $s,t \in [0,1]$.}

Let $\Xi_{s,t}: T_{f_t(p)}M \to T_{f_t(p)}M, s,t \in [0,1],$ be a family of isomorphisms which is equal to the identity for $(s,t) \in \{0\} \times [0,1] \cup [0,1] \times \{0,1\}$ and satisfies $\Xi_{s,t} \circ D_p f_t = F_s^t$ for all $s,t \in [0,1]$. Such $\Xi_{s,t}$ exist since the map $GL_{2n+1}(\R) \to \text{Mono}_\R(\R^n,\R^{2n+1})$ given by restriction to the first $n$ coordinate directions is a Serre fibration. Choose a local coordinate chart around $p$ diffeomorphic to $\R^{2n+1}$ in which $f_t(p)$ is identified with the origin. Let $\Phi_{s,t}: \R^{2n+1} \to \R^{2n+1}, s,t \in [0,1],$ be a family of compactly supported diffeomorphisms that agrees with $\Xi_{s,t}$ near the origin (where we identify $T_{f_t(p)}M$ with $\R^{2n+1}$ using the local coordinates) and is equal to the identity for $(s,t) \in \{0\} \times [0,1] \cup [0,1] \times \{0,1\}$. Such a $\Phi_{s,t}$ can be constructed by cutting off the generating vector field of $\Xi_{s,t}$ for fixed $t$ (viewed as an isotopy with time-parameter $s$) outside of a compact neighbourhood of the origin. Using the chosen coordinates, $\Phi_{s,t}$ may be viewed as a family of diffeomorphisms fixing $f_t(p)$ with compact support near $p$ which is equal to the identity for $(s,t) \in \{0\} \times [0,1] \cup [0,1] \times \{0,1\}$ and which satisfies $D_{f_t(p)} \Phi_{s,t} = \Xi_{s,t}$. Then $\Phi_{u,t} \circ f_t, t, u \in [0,1],$ defines an isotopy of families of embeddings starting at $f_t$ and ending at $\Phi_{1,t} \circ f_t$. It is well-known and easy to check that $\Phi_{u,t} \circ f_t$ extends to a family of formal embeddings $(\Phi_{u,t} \circ f_t, F^{t,u}_{s}), s,t,u \in [0,1],$ rel. $t \in \{0,1\}$ starting at $(f_t, D_p f_t *_s F_{s}^t)$ which can be chosen so that $F^{t,u}_{s}(p) = ((D_{f_t(p)} \Phi_{u(1-s),t} \circ D_{f_t(p)} \Phi_{u,t}^{-1}) \circ (D_p (\Phi_{u,t} \circ f_t)) *_s  F_{s}^t(p) = (D_{f_t(p)} \Phi_{u(1-s),t} \circ D_p f_t)) *_s  F_{s}^t(p) = (\Xi_{u(1-s),t} \circ D_p f_t) *_s  F_{s}^t(p) = F_{u(1-s)}^t *_s F_{s}^t(p)$. For $u = 1$, this becomes $F^t_{1-s}(p) *_s F^t_s(p): T_p L \to T_{f_t(p)} M$ which is clearly homotopic rel. $(s,t) \in \partial ([0,1]^2)$ to the constant (in $s$) family $F^t_1 = \Xi_{1,t} \circ D_p f_t = D_p (\Phi_{1,t} \circ f_t)$. From this it follows that $F_s^{t,1}$ is homotopic (with fixed underlying embedding) rel. $(s,t) \in \partial ([0,1]^2)$ to a family $\widetilde{F}^t_s,s,t \in [0,1],$ which satisfies $\widetilde{F}^t_s(p) = D_p (\Phi_{1,t} \circ f_t)$ for all $s,t\in [0,1]$. This finishes the third step as $(f_t, D_p f_t *_s F_{s}^t)$ is clearly homotopic to $(f_t, F_{s}^t)$ rel. $(s,t) \in \partial ([0,1])^2$.\\

\emph{Step 4. We may assume that $D_pf_t$ does not depend on $t$.}

Let $k_t:U \to M$ be a family of Legendrian embeddings of a neighbourhood $U \subseteq L$ of $p$ such that $k_t(p) = f_t(p)$ and $D_p k_t = D_p f_t$. Then there exists a contact isotopy $\Psi_t$ with compact support near $f_t(p)$ so that $k_0 = \Psi_t \circ k_t$ near $p$. Now $u \mapsto (\Psi_{ut} \circ f_t, D\Psi_{ut} \circ F^t_s)$ defines a homotopy of formal Legendrian isotopies. For $u = 1$ and at $p$ this becomes $(\Psi_{t} \circ f_t, D\Psi_{t} \circ F^t_s)(p) = (k_0(p), D \Psi_t \circ D_p f_t) = (k_0(p), D_p k_0)$ by Step 2 and our choice of $\Psi_t$ and $k_t$.\\

\emph{Step 5. Proof of the lemma.}

Let $\widetilde{f}:U \to M$ be a Legendrian embedding of a neighbourhood $U \subseteq L$ of $p$ that satisfies $D_p \widetilde{f} = D_p f_t$. On a sufficiently small neighbourhood $V \subseteq U$ of $p$, $\widetilde{f}$ is $C^1$-close to $f_t$ and $D \widetilde{f}$ is $C^0$-close to $F^t_s$ for all $s,t \in [0,1]$ by Step 4. The first observation implies that we can find a family of $C^1$-small isotopies $\Xi_{s,t}:M \to M, s,t \in [0,1]$, with compact support near $f_t(p)$ so that $\widetilde{f} = \Xi_{1,t} \circ f_t$ near $p$ for all $t \in [0,1]$, and $\Xi_{s,t}(p) = p$ and $D_p \Xi_{s,t} = id$ for all $s,t \in [0,1]$. Then $(\Xi_{1,t} \circ f_t, D(\Xi_{1-s,t} \circ f_t) *_s F^t_s)$ defines a formal Legendrian homotopy which is clearly homotopic to $(f_t, F^t_s)$, where we have again identified\footnote{As in the proof of the third step, we can use an identification of the tangent fibres in which the contact structure is constant near $p$.} the tangent fibres using local coordinates near $p$ so that we can view $D(\Xi_{1-s,t} \circ f_t) *_s F^t_s$ as maps of tangent bundles covering $\Xi_{1,t} \circ f_t$. Near $p$, $D(\Xi_{1-s,t} \circ f_t) *_s F^t_s$ is $C^0$ close to $D(\Xi_{1,t} \circ f_t)$. 

Since the space of monomorphisms $T_xL \to T_{\Xi_{1,t} \circ f_t(x)} M$ forms a smooth manifold and the Legendrian monomorphisms (i.e. those whose image is contained in $\xi$ and Lagrangian) form a smooth submanifold, all depending smoothly on $x$ and $t$, this implies that we can find a homotopy $\mathcal{F}^{u,t}_s, s,t,u \in [0,1],$ with support near $p$ starting at $D(\Xi_{1-s,t} \circ f_t) *_s F^t_s$ covering $\Xi_{1,t} \circ f_t$ with $\mathcal{F}^{u,t}_0 = D(\Xi_{1,t} \circ f_t)$ and $\mathcal{F}^{u,t}_1$ Legendrian for all $t,u \in [0,1]$ so that $\mathcal{F}^{1,t}_s = D(\Xi_{1,t} \circ f_t)$ near $p$ for all $s,t \in [0,1]$. 

This finishes the proof since $\Xi_{1,t} \circ f_t$ is Legendrian near $p$ by construction.
\eproof\\

\begin{rmk}\label{rmk:formal leg leg near a point}
	As is clear from the proof, Lemma \ref{lem:formal leg leg near a point} also holds if $(f_i,F_s^i)$ is not Legendrian for $i = 0$ or $i = 1$ and we don't demand that the homotopy from $(f_i,F_s^i)$ to $(g_i,G_s^i)$ is constant.\\
\end{rmk}

We will also need the following statement that allows us to approximate a formal Legendrian by a loose Legendrian.

\begin{lem}\label{lem:loose_legendrians_are_dense} \hspace{-1mm}\textnormal{\cite{mur12}}\,
	Let $(f:L \to M,F_s)$ be a formal Legendrian embedding of a closed and connected manifold $L$ into a contact manifold $(M,\xi)$ with $\dim L \geq 2$. Then $(f,F_s)$ is $C^0$-closely formally isotopic to a loose Legendrian embedding. This isotopy can be chosen to be constant outside of any non-empty neighbourhood of the set where $(f,F_s)$ is not genuinely Legendrian.
\end{lem}

\bproof We explain how this is a consequence of the proof of Corollary 5.1 in \cite{mur12} in a way similar to how Theorem 3.1 above is a consequence of the proof of Theorem 1.2 in \cite{mur12}. Let $\eps > 0$. As in \cite{mur12}, we may assume that there exists an open set $U \subseteq M$ of diameter smaller than $\eps$ so that $f$ is Legendrian on $f^{-1}(U)$ and that $(U,U \cap f(L))$ is a loose chart (see Remark \ref{rmk:formal leg leg near a point} above). By Proposition 3.4 in \cite{mur12}, there exists a formal homotopy $\eps$-close to $f$ which fixes the loose chart from $(f,F_s)$ to a wrinkled Legendrian embedding $g$. We can find pairwise disjoint loose charts inside of $U$, one for each wrinkle of $g$. By replacing each of those loose charts by an inside-out wrinkle, we can find a wrinkled Legendrian embedding $\widetilde{g}$ $\eps$-close to $g$ which admits markings for its wrinkles. Using the markings to resolve the wrinkles, we obtain a Legendrian embedding which is $3 \eps$-closely formally isotopic to $f$. Since Proposition 3.4 in \cite{mur12} holds relatively, and the other constructions in the proof are compactly supported, the lemma follows.
\eproof\\

\section{Proof of Theorem \ref{thm:main result}}\label{sec:proofs}

In this section we present the proof of our main result as outlined in the introduction. The idea of the proof is taken from the proof of Theorem 1.8 in \cite{drs20}.\\ 

\bproofof{Theorem \ref{thm:main result}} 
	
	Let $U_0,U_1 \subseteq M$ be two open subsets and $f_t$ a homotopy of Legendrian embeddings as in the statement of Theorem \ref{thm:main result}. Let $\Phi_t:M \to M$ be a contact isotopy with $\Phi_t \circ f_0 = f_t$. Let $\delta, \eta > 0$ be two positive numbers. 
	
	Let $W_i$ denote a Weinstein neighbourhood of $f_i(L)$ of height $\frac{\eta}{3}$ with fibers of diameter $\frac{\delta}{3}$ for $i =1,2$. By choosing $W_0$ close enough to $f_0(L)$, we may assume that $\Phi_1(W_0) \subseteq W_1$. 
	
	We perturb $f_0$ to get a formal Legendrian homotopy $(g_t: L \to W_0,G_{s,t})$ from $f_0$ to a formal Legendrian $(g_1,G_{s,1})$ which is non-Legendrian almost everywhere so that $g_t$ is $\frac{\delta}{6}$-close to $f_0$ and $\Phi_t \circ g_1$ is $\frac{\delta}{3}$-close to $f_t$ for all $t \in [0,1]$. Denote $\widetilde{L} \coloneqq g_1(L)$.
	
	According to Corollary \ref{cor:c^0_close_hofer_distance_for_non-legendrians}, there exist $t_0=0<...<t_N=1$, $N \in \N$, and a contact isotopy $\widetilde{\phi}_t$ with $\widetilde{\phi}_1 \circ g_1 = \Phi_1 \circ g_1$ and $\Vert \widetilde{\phi}_t \Vert_\alpha < \frac{\eta}{3}$ so that $\widetilde{\phi_t} \circ g_1$ is $\frac{\delta}{3}$-close to $\Phi_t \circ g_1$ for all $t \in [0,1]$. In particular, $\widetilde{\phi}_t \circ g_1$ is $\frac{2\delta}{3}$-close to $f_t$ for all $t \in [0,1]$.

	For any $\widehat{\delta} > 0$, we can find a formal Legendrian homotopy $(\chi_t: L \to W_0, X_{s,t})$ from a loose Legendrian embedding $\chi_0:L \to W_0$ to $(g_1,G_{s,1})$ so that $\chi_t$ is $\widehat{\delta}$-close to the constant homotopy (see Lemma \ref{lem:loose_legendrians_are_dense}). Denote $\widehat{L} \coloneqq \chi_0(L)$. We may choose $\chi_0$ in such a way that $\widehat{L}$ and $\widetilde{\phi}_t(\widehat{L})$ have loose charts of diameter smaller than $\frac{\delta}{3}$ contained in $W_0$ and $W_1$, respectively. We choose $\widehat{\delta} < \frac{\delta}{6}$ so small that $\widetilde{\phi}_t \circ \chi_0$ is $\frac{\delta}{3}$-close to $\widetilde{\phi}_t \circ g_1$ for all $t \in [0,1]$. In particular, $\widetilde{\phi}_t \circ \chi_0$ is $\delta$-close to $f_t$ for all $t \in [0,1]$.
	
	\begin{figure}
		\centering
		\includegraphics[width=\textwidth,height=0.9\textheight,keepaspectratio]{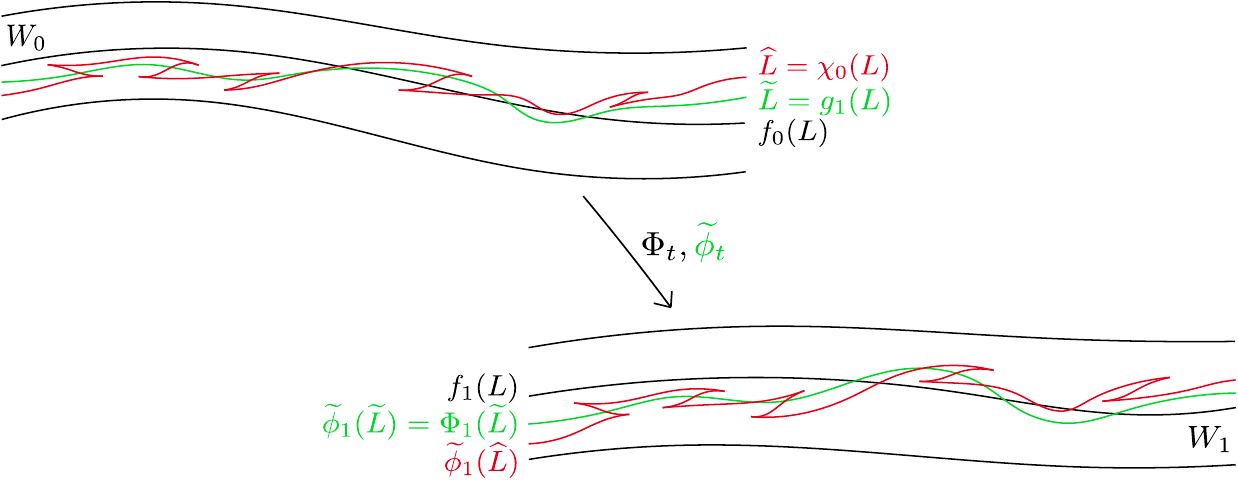}
		\caption{Starting from an isotopy $f_t$ of Legendrian embeddings, we construct a non-Legendrian perturbation $g_1$, a loose Legendrian approximation $\chi_0$ of $g_1$ inside of $W_0$, and contact isotopies $\Phi_t$ and $\widetilde{\phi}_t$ that satisfy $\Phi_t \circ f_0 = f_t$ and $\widetilde{\phi}_1(\widetilde{L}) = \Phi_1(\widetilde{L})$.}
		\label{fig:thm1-2pf}
	\end{figure}  
	
	We claim that we may assume that there exists a formal Legendrian homotopy $(\xi_t:L \to W_1,\Xi_{s,t})$ from $\widetilde{\phi}_1 \circ \chi_0$ to $f_1$ that is $\frac{\delta}{3}$-close to $f_1$. In order to see this, let $\widetilde{\delta} > 0$ be so small that for all points $x \in f_t(L), t \in [0,1],$ and $y \in M$ so that the distance between $x$ and $y$ is smaller than $\widetilde{\delta}$, the distance between $(\Phi_1 \circ \Phi_t^{-1})(x)$ and $(\Phi_1 \circ \Phi_t^{-1})(y)$ is smaller than $\frac{\delta}{3}$ and $(\Phi_1 \circ \Phi_t^{-1})(y) \in W_1$. For any $\delta > 0$, we have constructed a formal Legendrian homotopy $(g_t,G_{s,t})*(\chi_{1-t},X_{s,1-t})$ from $f_0$ to $\chi_0$ which is $\frac{\delta}{3}$-close to $f_0$, and we proved that we can find a Legendrian isotopy $\widetilde{\phi_t} \circ \chi_0$ from $\chi_0$ to $\widetilde{\phi}_1 \circ \chi_0$ which is $\delta$-close to $f_t$. We can apply the above constructions\footnote{Note that we have not used the properties of $W_1$ in our constructions yet. This implies that we can apply those constructions to a smaller $\delta$ while keeping $W_1$ fixed.} to some $\delta$ smaller than $\widetilde{\delta}$ so that we may assume that $(g_t,G_{s,t})*(\chi_{1-t},X_{s,1-t})$ and $\widetilde{\phi_t} \circ \chi_0$ are, in fact, $\widetilde{\delta}$-close to $f_0$ and $f_t$, respectively. Then by our choice of $\widetilde{\delta}$, the formal Legendrian homotopy $\left( \Phi_1 \circ \Phi_{1-t}^{-1} \circ (\widetilde{\phi}_{1-t}\circ \chi_0) \right) * \Big( (\Phi_1, D \Phi_1) \circ ((\chi_{t},X_{s,t})*(g_{1-t},G_{s,1-t})) \Big)$ from $\widetilde{\phi}_1 \circ \chi_0$ to $f_1$ is $\frac{\delta}{3}$-close to $f_1$, and the underlying embeddings have image contained in $W_1$. This proves the claim.
	
	For $i \in \{0,1\}$, let $V_i$ be a relatively compact subset of $U_i$ so that $(V_i, f_i(L) \cap V_i)$ is a loose chart for $f_i(L)$. Let $(h^i_t:L \to W_i,H^i_{s,t})$ be a formal Legendrian homotopy $\frac{\delta}{3}$-close to $f_i$ from $f_i$ to a loose Legendrian embedding $h^i_1: L \to W_i$ which admits a loose chart contained in $W_i \cap V_i$ whose diameter is smaller than $\frac{\delta}{3}$, which is constant outside of a compact subset of $f_i^{-1}(f_i(L) \cap W_i \cap V_i)$ and maps $f_i^{-1}(f_i(L) \cap W_i \cap V_i)$ into $W_i \cap V_i$ (see Figure \ref{fig:thm1-2pf3}). Such $(h^i_t,H^i_{s,t})$ may be found for example by performing a $(\chi = 0)$-stabilization of $f_i$.
	
	\begin{figure}
		\centering
		\includegraphics[width=\textwidth,height=0.9\textheight,keepaspectratio]{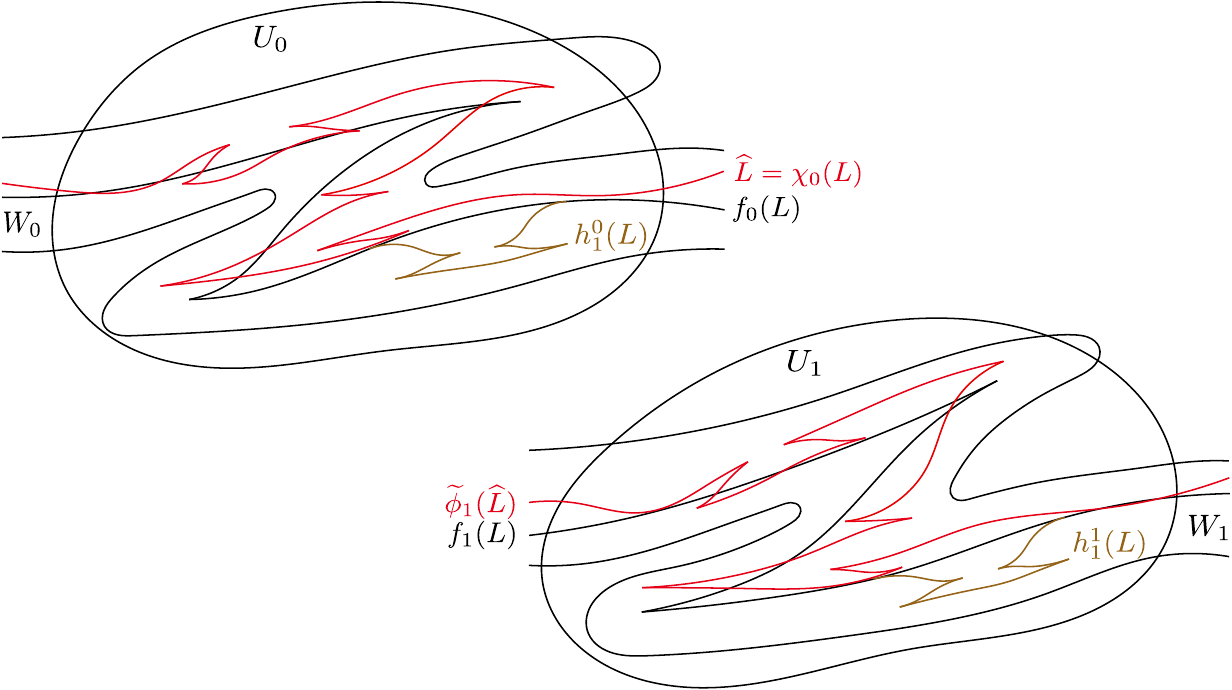}
		\caption{Inside of $V_i \cap W_i$, we find a formal Legendrian homotopy $(h^i_t,H^i_{s,t})$ from $f_i$ to a loose Legendrian embedding $h^i_1$ which admits a small loose chart.}
		\label{fig:thm1-2pf3}
	\end{figure}  
	
	By applying Murphy's h-principle for $i \in \{0,1\}$ to the formal Legendrian homotopies $(h^i_t,H^i_{s,t})$ inside of $V_i \subseteq U_i$, we find contact isotopies $\psi^i_t:M \to M$ with compact support in $U_i$ so that $\psi^0_1 \circ f_0 = h^0_1$ and $\psi^1_1 \circ h^1_1 = f_1$. By definition of $U_i$, we may assume that $\Vert \psi^i_t \Vert_\alpha < \eps_i$. 
	
	As $h^i_1(L)$, $\widehat{L}$, and $\widetilde{\phi}_1(\widehat{L})$ have loose charts of diameter smaller than $\frac{\delta}{3}$, we can apply the $C^0$-close version of Murphy's h-principle (Theorem \ref{thm:murphys h-principle}) to the formal Legendrian homotopies $(h^0_{1-t},H^0_{s,1-t})*(g_t,G_{s,t})*(\chi_{1-t},X_{s,1-t})$ and $(\xi_t,\Xi_{s,t})*(h^1_t,H^1_{s,t})$ to find contact isotopies $\theta^i_t:M \to M, i \in \{0,1\},$ with compact support in $W_i$ so that $\theta^0_t \circ h^0_1$ is a homotopy of Legendrian embeddings from $h^0_1$ to $\chi_0$ and $\theta^1_t \circ \widetilde{\phi}_1 \circ \chi_0$ is a homotopy of Legendrian embeddings from $\widetilde{\phi}_1 \circ \chi_0$ to $h^1_1$ which are $\frac{\delta}{3}$-close to $(h^0_{1-t},H^0_{s,1-t})*(g_t,G_{s,t})*(\chi_{1-t},X_{s,1-t})$ and $(\xi_t,\Xi_{s,t})*(h^1_t,H^1_{s,t})$, respectively.
	
	Note that $(h^0_{1-t},H^0_{s,1-t})*(g_t,G_{s,t})*(\chi_t,X_{s,t})$ and $(\xi_t,\Xi_{s,t})*(h^1_t,H^1_{s,t})$ are $\frac{\delta}{3}$-close to $f_0$ and $f_1$, respectively, and thus $\theta^0_t \circ h^0_1$ is $\frac{2\delta}{3}$-close to $f_0$, and $\theta^1_t \circ \widetilde{\phi}_1 \circ \chi_0$ is $\frac{2\delta}{3}$-close to $f_1$ for all $t \in [0,1]$.
		
	Since $W_i$ has height $\frac{\eta}{3}$ and fibers of diameter $\frac{\delta}{3}$, it follows from Proposition \ref{prop:jet-space_with_small_energies} applied to $\theta^i_t$ that we can find contact isotopies $\widetilde{\theta}^i_t:M \to M$ that satisfy $\widetilde{\theta}^i_1=\theta^i_1$ and $\Vert \widetilde{\theta}^i_t \Vert_\alpha < \frac{\eta}{3}$ so that $\widetilde{\theta}^0_t|_{h^0_1(L)}$ and $\widetilde{\theta^1_t}|_{\widetilde{\phi}_1 \circ \chi_0(L)}$ are $\frac{\delta}{3}$-close to $\theta^0_t|_{h^0_1(L)}$ and $\theta^1_t|_{\widetilde{\phi}_1 \circ \chi_0(L)}$, respectively. Note that then $\widetilde{\theta}^0_t \circ h^0_1 = \widetilde{\theta}^0_t \circ \psi^0_1 \circ f_0$ and $\widetilde{\theta}^1_t \circ \widetilde{\phi}_1 \circ \chi_0$ are $\delta$-close to $f_0$ and $f_1$, respectively, for all $t \in [0,1]$.

	Because $\widetilde{\phi}_t \circ \chi_0$ is $\delta$-close to $f_t$, $(\widetilde{\theta}^0 * \widetilde{\phi} * \widetilde{\theta}^1)_t \circ \psi^0_1 \circ f_0$ will be $\delta$-close to $f_t$ if the concatenation is performed in such a way that $\widetilde{\theta}^0_t$ and $\widetilde{\theta}^1_t$ are traversed very fast. Since also $\Vert \widetilde{\phi}_t \Vert_\alpha < \frac{\eta}{3}$, we see that $\Vert (\widetilde{\theta}^0 * \widetilde{\phi} * \widetilde{\theta}^1)_t \Vert_\alpha < \eta$.
	
	Then the isotopies $\psi^0_t$, $\psi^1_t$, and $\phi_t \coloneqq (\widetilde{\theta}^0 * \widetilde{\phi} * \widetilde{\theta}^1)_t$ have the desired properties.
\eproof\\

\section{Proofs of the Corollaries}\label{sec:proof of stabilized case}

In this section we present the proofs of Theorems \ref{thm:displacement of loose legendrians} and \ref{thm:displacement_below_action_cutoff}, and Corollaries \ref{cor:approximations_of_arbitrary_hamiltonians}, \ref{cor:isotopy_of_stabilized_legendrian}, \ref{cor:displacement_of_stabilized_legendrian}, and \ref{cor:quantitative h-principle}.\\

\bproofof{Theorem \ref{thm:displacement of loose legendrians} and Theorem \ref{thm:displacement_below_action_cutoff}}
Assume that the assumptions in Theorem $\ref{thm:displacement of loose legendrians}$ are satisfied. Let $L_1 \subseteq M$ be a closed Legendrian submanifold which is Legendrian isotopic to $L$ such that there are no Reeb chords between $L$ and $L_1$. In particular, $L_1$ is loose. It follows from Murphy's h-principle that $L_1$ is Legendrian isotopic to some Legendrian $L_2$ which agrees with $L_1$ outside of an arbitrarily small strict Darboux ball $U_1$ around some point in $L_1$, and so that $U_1 \cap L_2$ is loose in $U_1$. If $U_1$ is chosen sufficiently small, all Reeb chords between $L$ and $L_2$ will have action larger than $E$, there will be no Reeb chords at all if the image of $L$ under the Reeb flow is closed, and $U_1$ will satisfy the property in the statement of Theorem $\ref{thm:main result}$ for $\eps_1 = \frac{\eta}{2}$ by Proposition \ref{prop:jet-space_with_small_energies}. Now we can apply Theorem \ref{thm:main result} to find compactly supported contact isotopies $\psi_t^0$, $\phi_t,$ and $\psi_t^1$ such that $\psi_t^0$ has support in $U$, $\Vert \psi_t^0\Vert_\alpha < \eps$, $\Vert \phi_t \Vert_\alpha < \frac{\eta}{2}$, $\Vert \psi_t^1 \Vert_\alpha < \frac{\eta}{2}$, and $L_2 = (\psi_1^1 \circ \phi_1 \circ \psi_1^0)(L)$. In particular, the isotopies $\psi_t^0$ and $(\phi * \psi^1)_t$ have the desired properties.

The proof of Theorem \ref{thm:displacement_below_action_cutoff} works in the same way after we note that if all Reeb chords from $L$ to $L_1$ have action larger than $E_1$, and all Reeb chords from $L_1$ to $L$ have action larger than $E_2$, then also all Reeb chords from $L$ to $L_2$ have action larger than $E_1$, and all Reeb chords from $L_2$ to $L$ have action larger than $E_2$, as long as $L_2$ is sufficiently $C^0$-close to $L_1$ since \begin{equation}
	\bigcup_{t \in [-E_2,E_1]} \phi^\alpha_t(L)
\end{equation}
is closed as a subset of $M$, where $\phi^\alpha_t$ denotes the Reeb flow of $\alpha$.
\eproof\\ \\

\bproofof{Corollary \ref{cor:approximations_of_arbitrary_hamiltonians}}
Assume that the assumptions in Corollary \ref{cor:approximations_of_arbitrary_hamiltonians} are satisfied. Let $\phi^H_t$ denote the contact isotopy associated to the function $H_t$. Let $f_t:L \to M$ be a homotopy of Legendrian embeddings and assume that there exist open sets $U_0,U_1$ as in the statement of Theorem \ref{thm:main result} such that $f_i(L) \cap U_i \subseteq U_i$ is loose for $i=1,2$. For any $\widetilde{\eta}>0$, we can apply Theorem \ref{thm:main result} to the family $(\phi^H_t)^{-1} \circ f_t$ to find isotopies $\psi^0_t$, $\widetilde{\psi}^1_t$, and $\widetilde{\phi}_t$ with $\Vert \psi^0_t \Vert_\alpha < \eps_0$, $\supp(\psi^0_t) \subseteq U_0$, $\supp(\widetilde{\psi}^1_t) \subseteq (\phi^H_1)^{-1}(U_1)$, $\Vert \widetilde{\phi}_t \Vert_\alpha < \widetilde{\eta}$, and $\widetilde{\psi}^1_1 \circ \widetilde{\phi}_1 \circ \psi^0_1 \circ f_0 = (\phi^H_1)^{-1} \circ f_1$. Let $\widetilde{F}_t$ denote the contact Hamiltonian associated to $\widetilde{\phi}_t$. Then the contact Hamiltonian $F_t$ associated to $\phi_t \coloneqq \phi_t^H \circ \widetilde{\phi}_t$ is given by $F_t(x) = H_t(x) + h_t((\phi^H_t)^{-1}(x)) \widetilde{F}_t((\phi^H_t)^{-1}(x))$, where $h_t$ denotes the positive function, which is equal to $1$ outside of a compact set, defined by $(\phi^H_t)^* \alpha = h_t \alpha$. In particular, $\Vert H_t- F_t \Vert$ can be made arbitrarily small by decreasing $\widetilde{\eta}$. Define $\psi^1_t \coloneqq \phi^H_1 \circ \widetilde{\psi}^1_t \circ (\phi^H_1)^{-1}$. Then $\psi^1_t$ is supported in $U_1$, and we may assume that $\Vert \psi^1_t \Vert_\alpha < \eps_1$. It also follows that $\psi^1_1 \circ \phi_1 \circ \psi^0_1 \circ f_0 = f_1$. Furthermore, if we choose these isotopies in such a way that $\widetilde{\phi}_t \circ \psi^0_1 \circ f_0$ is sufficiently $C^0$-close to $(\phi^H_t)^{-1} \circ f_t$, then $\phi_t \circ \psi^0_1 \circ f_0$ is $C^0$-close to $f_t$.
\eproof\\ \\

Corollary \ref{cor:isotopy_of_stabilized_legendrian} and Corollary \ref{cor:displacement_of_stabilized_legendrian} are consequences of Theorem \ref{thm:main result}.\\

\bproofof{Corollary \ref{cor:isotopy_of_stabilized_legendrian}} Assume that $U_i$, $L_i$, $V_i$, $SL_i$, $i \in \{0,1\}$, and $f_t$ are as in the statement of Corollary \ref{cor:isotopy_of_stabilized_legendrian}. Let $\Phi_t: M \to M$ be a compactly supported  contact isotopy with $\Phi_t \circ f_0 = f_t$. Recall that $SL_i \cap U_i \subseteq U_i$ is loose. Let $\eta, \mu, \delta > 0$. We identify $U_i$ with a subset of the 1-jet bundle $J^1L$ with coordinates $(q,p,z)$, $q \in L,p \in T^*_qL, z \in \R,$ as in the definition of a local Weinstein neighbourhood and write the height explicitly as $U^i_{\eps_i} \coloneqq U_i$. For $\lambda \in (0,1]$, we denote by $U^i_{\lambda \eps_i} \subseteq U^i_{\eps_i}$ the image of $U^i_{\eps_i}$ under the contactomorphism $(q,p,z) \mapsto (q, \lambda p,\lambda z)$. We let $\widetilde{\eps}_i < \eps_i$ be such that $SL_i $ is still stabilized inside of $U^i_{\widetilde{\eps}_i}$. It follows from the proof of Proposition \ref{prop:jet-space_with_small_energies} that for any $\lambda \in (0,1]$ there exists a contact isotopy $\widetilde{\psi}^i_t$ with compact support\footnote{Technically speaking, we may also have to shrink $U^i_{\widetilde{\eps}_i}$ a bit in the $q$-coordinate in order to ensure that $\widetilde{\psi}^i_t$ can be chosen to have compact support in $U^i_{\eps_i}$, but this causes no issue and we omit to write this explicitly.} in $U^i_{\eps_i}$ that satisfies $\widetilde{\psi}^i_1(U^i_{\widetilde{\eps}_i}) \subseteq U^i_{\lambda \widetilde{\eps}_i}$ and $\Vert \widetilde{\psi}^i_t \Vert_\alpha < \widetilde{\eps}_i + \mu$ for $i \in \{0,1\}$. Choose $\mu$ and $\lambda$ so small that $\widetilde{\eps}_i + \mu + 2 \lambda \widetilde{\eps}_i < \eps_i$ for $i \in \{0,1\}$. According to Proposition \ref{prop:jet-space_with_small_energies}, $U^i_{\lambda \widetilde{\eps}_i}$ satisfies the property of $U_i$ in the statement of Theorem \ref{thm:main result} with the constant $2 \lambda \widetilde{\eps}_i$. Hence, we can apply Theorem \ref{thm:main result} to the homotopy $\left( (\widetilde{\psi}^0_{1-t} \circ (\widetilde{\psi}^0_1)^{-1}) * \Phi_t * \widetilde{\psi}^1_t \right) \circ \widetilde{\psi}^0_1 \circ f_0$ to conclude that there exist contact isotopies $\widehat{\phi}_t$ and $ \widehat{\psi}^i_t$, $i \in \{0,1\}$, such that 
\begin{equation}
\Vert \widehat{\phi}_t \Vert_\alpha < \min\left\{\eta,\frac{\eps_0 - (\widetilde{\eps}_0 + \mu + 2 \lambda \widetilde{\eps}_0)}{2}, \frac{\eps_1 - (\widetilde{\eps}_1 + \mu + 2 \lambda \widetilde{\eps}_1)}{2} \right\},
\end{equation}
$\Vert \widehat{\psi}^i_t \Vert_\alpha < 2 \lambda \widetilde{\eps}_i$, $\supp(\widehat{\psi}^i_t) \subseteq U^i_{\lambda \widetilde{\eps_i}}$, and $\widehat{\psi}^1_1 \circ \widehat{\phi}_1 \circ \widehat{\psi}^0_1 \circ \widetilde{\psi}^0_1 \circ f_0 = \widetilde{\psi}^1_1 \circ f_1$. Furthermore, we can assume that $\widehat{\phi}_t \circ \widehat{\psi}^0_1 \circ \widetilde{\psi}^0_1 \circ f_0$ is $\delta$-close to $((\widetilde{\psi}^0_{1-s} \circ (\widetilde{\psi}^0_1)^{-1})) * \Phi_s * \widetilde{\psi}^1_s)_t \circ \widetilde{\psi}^0_1 \circ f_0$. This concatenation is performed in such a way that $\widetilde{\psi}^0_{1-t} \circ (\widetilde{\psi}^0_1)^{-1}$ is traversed during the time interval $[0,1/3]$, $\Phi_t$ during the time interval $[1/3,2/3]$, and $\widetilde{\psi}^1_t$ during the time interval $[2/3,1]$. As $\widetilde{\psi}^i_t(SL_0)$ is contained in $V_i$ by construction, we may assume after potentially using appropriate cut-offs outside of a compact subset of $V_i$ that
$\{\widehat{\phi}_{t/3}\}_{t \in [0,1]}$ and $\{\widehat{\phi}_{2/3 + t/3} \circ (\widehat{\phi}_{2/3})^{-1}\}_{t \in [0,1]}$ have compact support contained in $V_0$, resp. $V_1$, as long as $\widehat{\phi}_t \circ \widehat{\psi}^0_1 \circ \widetilde{\psi}^0_1 \circ f_0$ is sufficiently close to $((\widetilde{\psi}^0_{1-s} \circ (\widetilde{\psi}^0_1)^{-1})) * \Phi_s * \widetilde{\psi}^1_s)_t \circ \widetilde{\psi}^0_1 \circ f_0$ so that $\widehat{\phi}_{t/3} \circ \widetilde{\psi}^0_1 (SL_0) \subseteq V_0$ and $\widehat{\phi}_{2/3+t/3} \circ \widetilde{\psi}^0_1 (SL_0) \subseteq V_1$ for all $t \in [0,1]$.

Now, we define $\psi^0_t = (\widetilde{\psi}^0_s * \widehat{\psi}^0_s * \widehat{\phi}_{s/3})_t$, $\psi^1_t = \left(((\widehat{\phi}_{2/3 + s/3} \circ (\widehat{\phi}_{2/3})^{-1}) * \widehat{\psi}^1_s * (\widetilde{\psi}^1_{1-s} \circ (\widetilde{\psi}^1_1)^{-1})\right)_t$, and $\phi_t = \widehat{\phi}_{1/3 + t/3} \circ (\widehat{\phi}_{1/3})^{-1}$, $t \in [0,1]$. Here, $s \in [0,1]$ for all of the isotopies in the concatenations. It is now straightforward to check that these maps have the desired properties. \eproof\\ \\

\bproofof{Corollary \ref{cor:displacement_of_stabilized_legendrian}} 
The proof of Corollary \ref{cor:displacement_of_stabilized_legendrian} combines the proofs of Corollary \ref{cor:isotopy_of_stabilized_legendrian} and Theorem \ref{thm:displacement of loose legendrians}. 

Let $L^n \subseteq M$, $n \geq 2$, be a closed Legendrian submanifold and $U_\eps$ a local Weinstein neighbourhood of $L$ of height $2\eps>0$. Let $SL$ be a displaceable stabilized version of $L$ such that the stabilization is performed in $U_\eps$. Again, $SL \cap U_\eps \subseteq U_\eps$ is loose. As above, we denote by $U_{\lambda \eps} \subseteq U_\eps$ the image of $U_\eps$, viewed as a subset of $J^1 L$, under the contactomorphism $(q,p,z) \mapsto (q,\lambda p, \lambda z)$ for $\lambda \in (0,1]$. Let $\widetilde{\eps} < \eps$ be such that $SL$ is stabilized inside of $U_{\widetilde{\eps}}$. Let $\mu > 0$. It follows as before that for any $\lambda \in (0,1]$ there exists a compactly supported contact isotopy $\psi_t$ with $\psi_1(U_{\widetilde{\eps}}) \subseteq U_{\lambda \widetilde{\eps}}$ and $\Vert \psi_t \Vert_\alpha < \widetilde{\eps} + \mu$. According to Proposition \ref{prop:jet-space_with_small_energies}, $U_{\lambda \widetilde{\eps}}$ satisfies the property of $U$ in the statement of Theorem \ref{thm:main result} with the constant $2 \lambda \widetilde{\eps}$. 

Let $E > 0$, and choose $\mu$ and $\lambda$ so small that $\widetilde{\eps} + \mu + 2 \lambda \widetilde{\eps} < \eps$. Let $L_1$ be a Legendrian which is Legendrian isotopic to $SL$ so that there are no Reeb chords between $L_1$ and $SL$. In particular $L_1$ is loose, and it is Legendrian isotopic to $L_2$, where $L_2$ is obtained from $L_1$ via a stabilization inside of some Darboux ball $U_1$. We may choose $U_1$ in such a way that there are no Reeb chords between $SL$ and $L_2$ of action larger than $E$ and no Reeb chords at all if the image of $SL$ under the Reeb flow is closed, and $U_1$ satisfies the property in the statement of Theorem \ref{thm:main result} with $\eps_1 = \frac{1}{2}(\eps - \widetilde{\eps} - \mu - 2 \lambda \widetilde{\eps}) > 0$. We now apply Theorem \ref{thm:main result} to a Legendrian isotopy from $\psi_1(SL)$ to $L_2$, to conclude that there exist contact isotopies $\psi^0_t$, $\psi^1_t$, and $\phi_1$ so that $\psi^1_1 \circ \phi_1 \circ \psi^0_1 (\psi_1(SL)) = L_2$ and $\Vert \psi^0_t \Vert_\alpha < 2 \lambda \widetilde{\eps}$, $\Vert \phi_t \Vert_\alpha < \frac{1}{2}(\eps - \widetilde{\eps} - \mu - 2 \lambda \widetilde{\eps})$, and $\Vert \psi^0_t \Vert_\alpha < \frac{1}{2}(\eps - \widetilde{\eps} - \mu - 2 \lambda \widetilde{\eps})$. Then the concatenation $\psi_t * \psi^0_t * \phi_t * \psi^1_t$ has the desired properties. \eproof\\ \\

\bproofof{Corollary \ref{cor:quantitative h-principle}}
Let $L_0$ and $L_1$ be two closed loose Legendrian submanifolds of $M$ that are formally isotopic and admit loose charts of size $\eps_0$ and $\eps_1$, respectively. By Murphy's h-principle for loose Legendrians, $L_0$ and $L_1$ are Legendrian isotopic. Let $\eta > 0$, and let $V \subseteq M$ be an open subset which satisfies the property of $U_i$ in the statement of Theorem \ref{thm:main result} with $\eps_i$ (in Theorem \ref{thm:main result}) equal to $\frac{\eta}{3}$. By definition of the size of a loose chart, there exist contact isotopies $\psi^i_t, i \in \{0,1\},$ with $\Vert \psi^i_t \Vert_\alpha \leq \frac{\eps_i}{2}$ so that $\psi^i_1(L_i)$ has a loose chart contained in $V$. By applying Theorem \ref{thm:main result} to $\psi^0_1(L_0)$ and $\psi^1_1(L_1)$, it follows that there exists a contact isotopy $\phi_t$ with $\phi_1 (\psi^0_1(L_0)) = \psi^1_1(L_1)$ and $\Vert \phi_t \Vert_\alpha < \eta$. Then the concatenation $\psi^0_t * \phi_t * ((\psi^1_{1-t})^{-1} \circ \psi^1_1)$ has the desired properties. \eproof\\

\bibliography{references}

\providecommand{\bysame}{\leavevmode\hbox to3em{\hrulefill}\thinspace}
\providecommand{\MR}{\relax\ifhmode\unskip\space\fi MR }
\providecommand{\MRhref}[2]{%
  \href{http://www.ams.org/mathscinet-getitem?mr=#1}{#2}
}
\providecommand{\href}[2]{#2}
\begin{thebibliography}{DRS22}

\bibitem[Ban97]{ban97}
Augustin Banyaga, \emph{The structure of classical diffeomorphism groups},
  first ed., Mathematics and Its Applications, Springer, Boston, MA, 1997.

\bibitem[DRS20]{drs20}
Georgios Dimitroglou~Rizell and Michael Sullivan, \emph{The persistence of the
  {C}hekanov–{E}liashberg algebra}, Selecta Mathematica \textbf{26} (2020).

\bibitem[DRS21]{drs21}
\bysame, \emph{The persistence of a relative {R}abinowitz-{F}loer complex},
  preprint, arXiv:2111.11975v3, 2021.

\bibitem[DRS22]{drs22}
\bysame, \emph{${C}^0$-limits of {L}egendrians and positive loops}, preprint,
  arXiv:2212.09190v2, 2022.

\bibitem[Gei08]{gei08}
Hansjörg Geiges, \emph{An introduction to contact topology}, Cambridge Studies
  in Advanced Mathematics, Cambridge University Press, 2008.

\bibitem[Mur12]{mur12}
Emmy Murphy, \emph{Loose legendrian embeddings in high dimensional contact
  manifolds}, preprint, arXiv:1201.2245v5, 2012.

\bibitem[Oh21]{oh21}
Yong-Geun Oh, \emph{Geometry and analysis of contact instantons and
  entanglement of legendrian links i}, preprint, arXiv:2111.02597v2, 2021.

\bibitem[RZ18]{rz18}
Daniel Rosen and Jun Zhang, \emph{Chekanov’s dichotomy in contact topology},
  Mathematical Research Letters (2018).

\bibitem[She17]{she17}
Egor Shelukhin, \emph{The {H}ofer norm of a contactomorphism}, J. Symplectic
  Geom. \textbf{15} (2017), no.~4, 1173--1208. \MR{3734612}

\end{thebibliography}
\bibliographystyle{amsalpha}

\end{document}